\documentclass[11pt,a4paper]{amsart}
\usepackage{cancel}
\usepackage[normalem]{ulem}
\usepackage{amssymb}
\usepackage{latexsym}
\usepackage{exscale}
\usepackage{amsfonts}
\usepackage{graphicx}
\usepackage{mathrsfs}
\usepackage{amsmath,amscd,amsthm}
\usepackage{bbm,pifont}
\usepackage{enumerate}
\usepackage{color}
\usepackage{amssymb}
\usepackage{latexsym}
\usepackage{exscale}
\usepackage[utf8]{inputenc}
\usepackage[T1]{fontenc}
\usepackage{dsfont}
\usepackage[mathscr]{eucal}
\usepackage[dvipsnames]{xcolor}

\usepackage[colorlinks, citecolor=blue,pagebackref,hypertexnames=false]{hyperref}
\makeatletter
\renewcommand\@makefnmark{%
  \hbox{\@textsuperscript{\normalfont\color{black}\@thefnmark}}}
\makeatother

\allowdisplaybreaks

\DeclareMathAlphabet\mathcalbf{OMS}{cmsy}{b}{n}
\DeclareMathAlphabet\EuScript{U}{eus}{m}{n}
\DeclareMathAlphabet\EuScriptBold{U}{eus}{b}{n}

\numberwithin{equation}{section}
\newtheorem{theorem}{Theorem}[section]
\newtheorem{lemma}[theorem]{Lemma}

\newtheorem{proposition}[theorem]{Proposition}

\newtheorem{definition}[theorem]{Definition}

\newtheorem{remark}[theorem]{Remark}

\newcommand{\Cine}{\EuScript C_\epsilon}

\newcommand{\Sopo}{\EuScript S_\omega}
\newcommand{\SopO}{\EuScript S_{\Ot}}
\newcommand{\Op}{\Omega_p}
\newcommand{\Ot}{\Omega_2}

\newcommand{\Tinv}{ (\mathcal T^s_\epsilon)^{-1}}

\def\C{\mathbb C}

\begin{document}
\allowdisplaybreaks

\title[Commutator of the Cauchy--Szeg\H{o} Projection]
{The commutator of the Cauchy--Szeg\H{o} Projection \\ for domains in $\mathbb C^n$ with minimal smoothness:\\ 
 weighted regularity}

\author{Xuan Thinh Duong, Loredana Lanzani, Ji Li and Brett D. Wick}

\address{Xuan Thinh Duong, Department of Mathematics, Macquarie University, NSW, 2109, Australia}
\email{xuan.duong@mq.edu.au}

\address{Loredana Lanzani,  University of Bologna, Piazza di Porta S. Donato 5, 40126 Bologna, Italy}
\email{loredana.lanzani2@unibo.it}

\address{Ji Li, Department of Mathematics, Macquarie University, NSW, 2109, Australia}
\email{ji.li@mq.edu.au}

\address{Brett D. Wick, Department of Mathematics \& Statistics\\
         Washington University - St. Louis\\
         St. Louis, MO 63130-4899 USA
         }
\email{wick@math.wustl.edu}


\begin{abstract}
Let $D\subset\C^n$ be a bounded, strongly pseudoconvex domain whose boundary $bD$ satisfies the minimal regularity condition of class $C^2$, and let $\Sopo$  denote
the Cauchy--Szeg\H{o} projection defined with respect to (any)
positive continuous multiple $\omega$ of induced Lebesgue measure for the boundary of $D$. We
   characterize compactness and boundedness  (the latter with explicit bounds) of the commutator $[b, 
\Sopo]$ in the Lebesgue space $L^p(bD, \Op)$ where $\Omega_p$ is any measure  in the  Muckenhoupt class $A_p(bD)$, $1<p<\infty$.
  We next fix $p =2$ and we let  $\SopO$  denote the Cauchy--Szeg\H{o} projection defined with respect to (any) 
 measure $\Ot \in
  A_2(bD)$, which is the largest class of reference measures for which a meaningful notion of Cauchy-Leray measure may be defined.  
 We characterize
 boundedness and compactness  in $L^2(bD, \Ot)$ of the commutator 
 $\displaystyle{[b,\SopO]}$.
 \end{abstract}

\maketitle

\centerline{\em Dedicated to Jill Pipher}
\vskip0.1in

\section{Introduction}
This is a companion paper to the recent work \cite{DLLW} where the weighted Lebesgue-space regularity
problem was studied for the Cauchy-Szeg\H o projection $\Sopo$ of a 
 strongly pseudoconvex domain $D\Subset \mathbb C^n$ that satisfies the minimal regularity condition of class $C^2$. 
  The reference measure in the definition of $\Sopo$ is taken to be
 $\omega := \Lambda \, \sigma$ (any) bounded, positive continuous multiple of the induced Lebesgue measure $\sigma$ (and we henceforth refer to any such $\omega$ as a {\em Leray Levi-like measure}), whereas the measures $\Omega_p$ with respect to which the weighted $L^p(bD, \Op)$-regularity of $\Sopo$ is established 
 in \cite{DLLW}, 
 belong to the maximal
  class of the Muckenhoupt measures $\{A_p(bD)\}_{1<p<\infty}$.
 \vskip0.1in
 
 In this paper we study the behavior in  $L^p(bD, \Op)$ of the commutator $[b, \Sopo]$. Specifically, we identify suitable conditions on the symbol $b$ 
  for which regularity and compactness of $[b, \Sopo]$ occur in 
 $L^p(bD, \Op)$ for any Muckenhoupt measure $\Op$, $1<p<\infty$, and we provide explicit bounds, see  \eqref{Lp-est-1} and \eqref{Lp-est-2} below. Doing so will also require studying the commutator $[b, \Cine]$ for the family $\{\Cine\}_\epsilon$ of Cauchy-type integral operators that were studied in \cite{DLLW} and \cite{LS2017}. 
 To be precise, letting $[\Op]_{A_p}$ denote the $A_p$-character of $\Op$, we have
 
  \begin{theorem}\label{cauchy4} Let $D\subset \mathbb C^n$, $n\geq 2$, be  a bounded, strongly pseudoconvex domain of class $C^2$.
 The following hold for any $b\in L^2(bD, \sigma )$ and for any Leray Levi-like measure $\omega$:

$(1)$  if $b\in{\rm BMO}(bD,\sigma)$ then the commutator $[b, \Sopo]$ is bounded on  $L^p(bD, \Op)$ for any $1<p<\infty$ and any $A_p$-measure $\Op$, with
\begin{equation}\label{Lp-est-1}
\|[b, \Sopo]\|_{{L^p(bD, \Op)\to L^p(bD, \Op)}}\lesssim \|b\|_{{\rm BMO}(bD,\sigma)} [\Op]_{A_p}^{ 4\cdot\max\{1,{1\over p-1}\}},
\end{equation}
where the implied constant depends on $p$, $D$ and $\omega$ {but are independent of $\Op$}. 

\vskip0.05in
 \hskip.6cmConversely, if {  $[b, \Sopo]$ and $[b, \Cine](I-\Sopo)$} are bounded on  $L^p(bD, \Op)$ for some $p\in (1, \infty)$ and for some $A_p$-measure $\Op$, then the symbol $b$ is in ${\rm BMO}(bD,\sigma)$ with
 \begin{align}\label{Lp-est-2}
 \|b\|_{{\rm BMO}(bD,\sigma)} &\lesssim c_\epsilon [\Op]_{A_p}^{  \max\{1,{1\over p-1}\}} 
 \|[b,  \Sopo ]\|_{L^p(bD, \Omega_p)\to L^p(bD, \Omega_p)} \\
 &\qquad+\|[b,  \Cine ](I-\Sopo)\|_{L^p(bD, \Omega_p)\to L^p(bD, \Omega_p)},\nonumber
 \end{align}
where the implied constant depends on $p$, $D$, $\omega$ and $[\Op]_{A_p}$.

 \vskip0.1in

$(2)$ if  $b\in{\rm VMO}(bD,\sigma)$ then the commutator $[b, \Sopo]$ is compact on 
 $L^p(bD, \Op)$ for all $1<p<\infty$ and all $A_p$-measures $\Op$. Conversely, if $[b, \Sopo]$ and and $[b, \Cine](I-\Sopo)$ are compact on  $L^p(bD, \Op)$ for some $p\in (1, \infty)$ and for some $A_p$-measure $\Op$, then the symbol $b$ is in ${\rm VMO}(bD,\sigma)$.
\end{theorem}

Theorem \ref{cauchy4}  extends to the optimal setting (that is, to $D$ with minimal smoothness) seminal results of Coifman--Rochberg--Weiss \cite{CRW}
 and Krantz--S.Y.\,Li \cite{KL2}. We point out that

\begin{itemize}
\item[{ \tt(1.)}] The exponent $4$ in \eqref{Lp-est-1} can be sharpened to $3+\delta$ for any $\delta>0$ but it cannot be reduced to $3$ due to the minimal smoothness of the domain.
\item[{\tt (2.)}] In the two necessity arguments in the above result, we need to make extra assumptions on $[b, \Cine](I-\Sopo)$. This is because in our setting of
 minimal smoothness there is no kernel information for $\Sopo$ that would
 guarantee a ``non-degenerate condition'' that would then give that $b\in{\rm BMO}(bD,\sigma)$; our extra assumptions are needed to ensure the latter.
\end{itemize}

  As was the case for the study of $\Sopo$ in \cite{DLLW}, it turns out that extrapolation is an effective tool to study the commutator  $[b, \Sopo]$ even though there is no baseline $L^2$-regularity that is naturally satisfied by $[b, \Sopo]$ (in great contrast with the situation for $\Sopo$ alone).
  We anticipate that extrapolation is also effective for characterizing finer properties, such as the Schatten-$p$ norm of $[b, \Sopo]$, and plan to address these questions in future work; see Feldman--Rochberg \cite{FR} for a related result. 
\vskip0.1in

The notion of Cauchy-Szeg\H o projection may be extended to any reference measure in 
the Muckenhoupt class $A_2(bD)$ (namely for $p=2$) and we adopt the notation
$\SopO$; as customary in this theory, the operator $\SopO$ is naturally bounded on $L^2(bD, \O)$. We have the following result for the commutator $[b, \SopO]$:

  \begin{theorem}\label{cauchy2}
Let $D\subset \mathbb C^n$, $n\geq 2$, be  a bounded, strongly pseudoconvex domain of class $C^2$.
The following hold for any $b\in L^2(bD, \Ot )$:

\vskip0.1in
$(1)$ if $b\in{\rm BMO}(bD,\sigma)$ then the commutator $[b, \SopO]$ is bounded on $L^2(bD, \Ot)$ for any $A_2$-measure $\Ot$. Conversely, if $[b, \SopO]$ is bounded on $L^2(bD, \Ot)$ for some $A_2$-measure $\Ot$, then $b\in{\rm BMO}(bD,\sigma)$;
\vskip0.1in

$(2)$  if $b\in{\rm VMO}(bD,\sigma)$ then the commutator $[b, \SopO]$ is compact on  $L^2(bD, \Ot)$  for any $A_2$-measure $\Ot$. Conversely, if $[b, \SopO]$ is compact on $L^2(bD, \Ot)$ for some $A_2$-measure $\Ot$, then $b\in{\rm VMO}(bD,\sigma)$.
\vskip0.1in
\noindent All the implied constants depend solely on $D$ and $\Ot$.
\end{theorem}

 \vskip0.1in

A few remarks 
are in order. 
\vskip0.1in
$\bullet$ As discussed in \cite{DLLW}, in our setting of minimal regularity  the classical tools (pointwise estimates of the Cauchy--Szeg\H o kernel) are not available. Instead, one makes a comparison of  $\Sopo$ (which we recall is the orthogonal projection of $L^2(bD, \omega)$ onto the holomorphic Hardy space $H^2(bD, \Omega)$) with certain families of Cauchy-type integral operators  $\{\Cine\}_\epsilon$ which are bounded projections (albeit non-orthogonal) of $L^2(bD, \omega)$ onto
$H^2(bD, \omega)$ whose kernels are completely explicit. This comparison yields the identities
$$
\Sopo = \Cine\circ \bigg(I -\big(\Cine^*-\Cine\big)\bigg)^{-1} \quad \text{in}\quad L^2(bD, \omega),\qquad 0<\epsilon <\epsilon (D)
$$
which turn out to be suitable replacements for the (unavailable) pointwise estimates for the Cauchy-Szeg\H o kernel.
\vskip0.05in
$\bullet$ On the other hand, commutators are not projection operators, so the aforementioned comparison argument for $\Sopo$ and $\Cine$ does not immediately percolate to the commutators $[b, \Sopo]$ and $[b, \Cine]$.
Instead, the proof of Theorem \ref{cauchy2}
makes use of the following family of identites:
\begin{equation}\label{E:KS-comm-2}
 [ b, \SopO]  = \bigg([b ,  \EuScript C_\epsilon]\, +\, \EuScript S_{{\Omega_2}}
 \circ \big[b,  I-(\Cine^\dagger - \Cine) \big]
\bigg)\circ \bigg(I-(\Cine^\dagger - \Cine)\bigg)^{\!\!-1} \ \text{in}\ \ L^2(bD, \Ot)
\end{equation}
for any $A_2$-like measure $\Ot$ and for any $0<\epsilon< \epsilon (D)$ as above.
For the commutator of $[ b, \Sopo ]$ we obtain the more precise descriptions
\begin{equation}\label{E:KS-comm-p}
[ b, \Sopo ]  = \bigg(\![b , \EuScript C_\epsilon ] + \Sopo\circ  \big[b, I-(\Cine^\dagger - \Cine)\big]  - [b, \Sopo]\circ\big((\EuScript R^s_\epsilon)^\dagger - \EuScript R^s_\epsilon\big)\!\bigg) \circ  
\bigg(\!I-\big((\Cine^s)^\dagger-\Cine^s\big)\!\bigg)^{-1}
\end{equation}
in $L^2(bD, \omega)$, for any Leray Levi-like measure $\omega$ and for any $0<\epsilon< \epsilon (D)$, which lead to the explicit bounds in the conclusion of Theorem \ref{cauchy4}.
Identities \eqref{E:KS-comm-2} and \eqref{E:KS-comm-p} are proved in Section
 \ref{S:4}.

\vskip0.1in

$\bullet$ In the statement of 
 Theorem \ref{cauchy2} 
 we assume that the symbol
$b$ is in $L^2(bD, \Ot)$ rather than the larger class $L^1(bD, \Ot)$, because 
the former is the natural (i.e. maximal) function space where the Cauchy--Szeg\H o projection $\SopO$
is defined. The requirement that $b$ is in $L^2(bD, \Ot)$ 
is not restrictive
  because $D$ is bounded and of class $C^2$, and $A_p$-measures for such domains are absolutely continuous with respect to the Leray Levi-like measures, hence $\Op(bD)<\infty $ for any such measure for any $1<p<\infty$. It follows that ${\rm BMO}(bD,\sigma)\subset L^2(bD, \Ot)$ {for any $\Ot\in A_2(bD)$}, see \eqref{E:incl}.
\vskip0.1in

$\bullet$ The space {${\rm BMOA}(bD,\sigma)$  (resp. ${\rm VMOA}(bD,\sigma)$)}
is the proper subspace of ${\rm BMO}(bD,\sigma)$ (resp. ${\rm VMO}(bD,\sigma)$) obtained by changing the a-priori condition that $b\in L^1(bD, \sigma)$ with the stricter requirement that $b$ is in the holomorphic Hardy space {$H^1(bD, \sigma)$, see  \cite{Po}; by the above argument{\footnote{Indeed, if $f\in {\rm BMOA}(bD,\sigma)$ then $f\in L^2(bD, \sigma)$ by the above argument. Hence $f\in H^1(bD, \sigma)\cap L^2(bD, \sigma)$ and this implies that $f\in H^2(bD, \sigma)$, see \cite[Corollary 2]{LS2016}.}}, 
${\rm BMOA}(bD,\sigma)\subset H^2(bD, \Ot)$ for any $\Ot\in A_2(bD)$}. Changing the a-priori condition that $b\in L^2(bD, \Ot)$ to $b\in H^2(bD, \Ot)$ in Theorem 
 \ref{cauchy2} produces new statements that are true for $b\in {\rm BMOA}(bD, \sigma)$ (resp. $b\in {\rm VMOA}(bD, \sigma)$), with the same proof.
\vskip0.1in

$\bullet$ The $L^p(bD, \Omega_p)$-regularity and -compactness problems for $[b, \SopO]$, while meaningful, are, at present, unanswered for $p\neq 2$.

\subsection{Further results} It is clear from
\eqref{E:KS-comm-2} and \eqref{E:KS-comm-p}  that 
one also needs to prove
quantitative results for the Cauchy Leray integrals $\{\EuScript C_\epsilon\}_\epsilon$ that extend the scope of the earlier works \cite{LS2017} and \cite{DLLWW} from Leray Levi-like measures, to $A_p$-measures: these are \cite[Theorem
3.1;  Proposition 3.2]{DLLW} along with Theorem \ref{T:5.2}
in Section \ref{S:3} below.

\subsection{Organization of this paper.} In the next section we recall the necessary background.
  All the quantitative results pertaining to the Cauchy--Leray integral are collected in 
 in Section \ref{S:3}. 
 Theorem \ref{cauchy4}, and Theorem \ref{cauchy2}, are proved in 
 Section \ref{S:4}.

\vskip0.1in

\section{Background}
 \label{S:2}
\setcounter{equation}{0}

In this section we introduce notations and recall earlier results \cite{DLLWW, LS2017}. For the reader's convenience, 
we also reproduce a few basic facts
 from  the companion paper \cite{DLLW} that will be used throughout this paper.
 We will henceforth assume that $D\subset\mathbb C^n$ is 
a bounded, strongly pseudoconvex domain of class $C^2$; that is, there is
$\rho\in C^2(\mathbb C^n, \mathbb R)$ which is strictly plurisubharmonic and such that
$D=\{z\in\mathbb C^n: \rho(z)<0\}$ and 
$bD=\{w\in\mathbb C^n: \rho(w)=0\}$ with $\nabla \rho(w)\not=0$ for all $w\in bD$.
 (We refer to such $\rho$ as a {\em defining function for $D$}; see e.g., \cite{Ra} for the basic properties of defining functions. Here we assume that one such $\rho$ has been fixed once and for all.) We will throughout make use of the following abbreviated notations:
 $$
\|T\|_p\ \equiv\ \|T\|_{L^p(bD, d\mu)\to L^p(bD, d\mu)},\quad\mbox{and}\quad  \|T\|_{p, q} \ \equiv\ \|T\|_{L^p(bD, d\mu)\to L^q(bD, d\mu)} 
 $$
 where the operator $T$ and the measure $\mu$ will be clear from context.
\vskip0.1in

\noindent $\bullet$ {\em The Levi polynomial and its variants.}\quad Define
$$ \mathcal L_0(w, z) := \langle \partial\rho(w),w-z\rangle -{1\over2} \sum_{j,k} {\partial^2\rho(w) \over \partial w_j \partial w_k} (w_j-z_j)(w_k-z_k),  $$
where $\partial \rho(w)=({\partial\rho\over\partial w_1}(w),\ldots, {\partial\rho\over\partial w_n}(w))$
and we have used the notation $\langle\eta,\zeta\rangle=\sum_{j=1}^n\eta_j\zeta_j$ for $\eta=(\eta_1,\ldots, \eta_n), \zeta=(\zeta_1,\dots,\zeta_n)\in\mathbb C^n$.
The strict plurisubharmonicity of $\rho $ implies that
$$  2\operatorname{ Re} \mathcal L_0(w, z) \geq -\rho(z)+c |w-z|^2,  $$
for some $c>0$, whenever $w\in bD$ and $z\in \bar D$ is sufficiently close to $w$.
We next define
\begin{align}\label{g0}
    g_0(w,z) := \chi \mathcal L_0+ (1-\chi) |w-z|^2
\end{align}
where $\chi=\chi(w,z)$ is a $C^\infty$-smooth cutoff function with $\chi=1$ when $|w-z|\leq \mu/2$ and $\chi=0$ if $|w-z|\geq \mu$.
Then for $\mu$ chosen sufficiently small (and then kept fixed throughout), we have that
\begin{equation}\label{E:c}
 \operatorname{ Re}g_0(w,z)\geq c(-\rho(z)+ |w-z|^2)
 \end{equation}
for $z$ in $\bar D$ and $w$ in $bD$, with $c$  a positive constant; we will refer to $g_0(w, z)$ as {\em the modified Levi polynomial}.
Note that $g_0(w, z)$ is polynomial in the variable $z$,  whereas in the variable $w$ it has no smoothness beyond mere continuity. To amend for this lack of regularity, 
 for each $\epsilon>0$ one considers a variant $g_\epsilon$ defined as follows.  Let $\{\tau_{jk}^\epsilon(w)\}$ be an $n\times n$-matrix  of $C^1$ functions such that
$$\sup_{w\in bD}\Big|{\partial^2\rho(w)\over\partial w_j\partial w_k}- \tau_{jk}^\epsilon(w)\Big|\leq\epsilon,\quad 1\leq j,k\leq n.$$
{  Set
\begin{align}\label{cepsilon}
c_\epsilon:=\sup_{w\in bD, 1\leq j,k\leq n} \big|\nabla \tau_{jk}^\epsilon(w)\big|.
\end{align}
For the convenience of our statement and proof, we may choose those $\{\tau_{jk}^\epsilon(w)\}$ such that 
\begin{align}\label{cepsilon bound}
c_\epsilon\lesssim \epsilon^{-1}.
\end{align}
where the implicit constant is independent of $\epsilon$.}
We also set
$$ \mathcal L_\epsilon(w, z) = \langle \partial\rho(w),w-z\rangle -{1\over2} \sum_{j,k}\tau_{jk}^\epsilon(w) (w_j-z_j)(w_k-z_k),  $$
and define
$$    g_\epsilon(w,z) = \chi \mathcal L_\epsilon+ (1-\chi) |w-z|^2, \quad z,w\in\mathbb C^n.  $$
Now $g_\epsilon$ is of class $C^1$ in the variable $w$, and
$$\left|g_0(w,z)-g_\epsilon(w,z) \right|\lesssim \epsilon |w-z|^2,\quad w\in bD, z\in \overline{D}.$$
We assume that $\epsilon$ is sufficiently small (relative to the constant $c$ in \eqref{E:c}), and this gives that
\begin{equation}\label{E:epsilon-0}
\left|g_0(w,z) \right|\leq  \left|g_\epsilon(w,z)\right|\leq \tilde C\left|g_0(w,z) \right|,\quad w, z\in bD
\end{equation}
where the constants $C$ and $\tilde C$ are independent of $\epsilon$; see \cite[Section 2.1]{LS2017}.

\vskip0.1in

\noindent $\bullet$ {\em The Leray--Levi measure for $bD$.}\quad 
We let $\sigma$ denote induced Lebesgue measure for $bD$ and we henceforth refer to the family
\begin{equation*}
\{\Lambda\sigma\}_\Lambda \equiv \left\{\omega:= \Lambda\, \sigma,\ \Lambda\in C(bD), \  0<c(D, \Lambda)\leq\Lambda(w)\leq C(D, \Lambda)<\infty\quad \text{for any}\ w\in bD\right\}
\end{equation*}
as the {\em Leray Levi-like measures}. This is because the Leray Levi measure $\lambda$, which plays a distinguished role in the analysis
 \cite{LS2017} of the Cauchy--Leray integrals 
$\{\Cine\}_\epsilon$ and their truncations $\{\Cine^s\}_\epsilon$, is a member of this family on account of the identity

 \begin{equation}\label{E:Leray Levi to sigma}
 d\lambda(w) =\Lambda(w)d\sigma(w),\quad w\in bD,
 \end{equation}
 
where $\Lambda \in C(\overline{D})$ satisfies the required bounds
$ 0< \epsilon(D) \leq \Lambda(w)\leq C(D)<\infty$ for any $ w\in bD$
as a consequence of the strong pseudoconvexity and $C^2$-regularity and boundedness of $D$. 
Hence we may equivalently express any Leray Levi-like measure $\omega$ as
\begin{equation}\label{E:LL-weights}
\omega = \varphi \lambda
\end{equation}
for some $\varphi \in C(\overline{bD})$ such that
$ 0< m(D) \leq \varphi(w)\leq M(D)<\infty$ for any $ w\in bD$.

Recall that (any) Leray-Levi measure $\lambda$ has density

 \begin{equation}\label{E:Leray Levi to sigma}
 d\lambda(w) =\Lambda(w)d\sigma(w),\quad w\in bD,
 \end{equation}

Then  the linear functional
\begin{align}\label{lambda}  
f\mapsto {1\over (2\pi i)^n} \int\limits_{bD} f(w) j^*(\partial \rho \wedge (\bar\partial \partial \rho)^{n-1})(w)=: \int\limits_{bD} f(w)d\lambda(w) 
\end{align}
where $f\in C(bD)$, defines  a measure $\lambda$ with positive density given by
$$   d\lambda(w) = {1\over (2\pi i)^n} j^*(\partial\rho \wedge (\bar\partial \partial\rho)^{n-1}) (w)  $$
where $j^*$ denotes the pullback under the inclusion $$j:bD\hookrightarrow\mathbb C^n.$$ 

We point out that the definition of $\lambda$ depends upon the choice of defining function for $D$, which here has been fixed once and for all; hence we refer to $\lambda$ as ``the'' {\em Leray--Levi measure}.
\vskip0.1in

\noindent $\bullet$ {\em A space of homogeneous type.}\quad Consider the function 
\begin{equation}\label{E:quasi-dist}
{\tt d}(w,z) := |g_0(w,z)|^{1\over2},\quad w, z\in bD.
\end{equation} 
It is known \cite[(2.14)]{LS2017} that
$$
|w-z|\lesssim {\tt d}(w, z)\lesssim |w-z|^{1/2},\quad w, z\in bD
$$
and from this it follows that the space of H\"older-type functions \cite[(3.5)]{LS2017}:
\begin{equation}\label{E:Holder}
|f(w) - f(z)|\lesssim {\tt d}(w, z)^\alpha\quad \mbox{for some}\ 0<\alpha\leq 1\ \ \mbox{and for all}\ w, z\in bD
\end{equation}
is dense in $L^p(bD, \omega)$, $1<p<\infty$ for any Leray Levi-like measure see \cite[Theorem 7]{LS2017}.

\vskip0.1in

It follows from \eqref{E:epsilon-0} that
\begin{align}\label{gd}
{\color{black} \tilde C{\tt d}(w,z)^{2}\leq |g_\epsilon(w,z)|\leq C {\tt d}(w,z)^{2}, \quad w, z\in bD}
 \end{align}
for any $\epsilon$ sufficiently small. It is shown in \cite[Proposition 3]{LS2017} that  ${\tt d}(w, z)$ is a quasi-distance: there exist constants $A_0>0$ and  $C_{\tt d}>1$ such that for all $w,z,z'\in bD$,
\begin{align}\label{metric d}
\left\{
                \begin{array}{ll}
                  1)\ \ {\tt d}(w,z)=0\quad {\rm iff}\quad w=z;\\[5pt]
                  2)\ \ A_0^{-1} {\tt d}(z,w)\leq  {\tt d}(w,z) \leq A_0 {\tt d}(z,w);\\[5pt]
                  3)\ \ {\tt d}(w,z)\leq C_{\tt d}\big( {\tt d}(w,z') +{\tt d}(z',z)\big).
                \end{array}
              \right.
\end{align}

\smallskip

Letting  $ B_r(w) $ denote the  boundary balls  determined via the quasi-distance ${\tt d }$,
\begin{align}\label{ball} 
B_r(w) :=\{ z\in bD:\ {\tt d}(z,w)<r \}, \quad {\rm where\ } w\in bD,
\end{align} 
we have that
\begin{align}\label{E:omegab}
c_\omega^{-1} r^{2n}\leq \omega\big(B_r(w) \big)\leq c_\omega r^{2n},\quad 0<r\leq 1,
\end{align}
for some $c_\omega>1$, see \cite[p. 139]{LS2017}.
It follows that the triples $\{bD, {\tt d}, \omega\}$, for any Leray Levi-like measure $\omega$,  are spaces of homogeneous type, where the measures $\omega$ have the doubling property:
{
\begin{lemma}\label{measure lambda}
The Leray Levi-like measures $\omega$ on $bD$ are doubling, i.e., there is a positive constant $C_\omega$ such that for all $x\in bD$ and $0<r\leq1$,
$$ 0<\omega(B_{2r}(w))\leq C_\omega\omega(B_{r}(w))<\infty. $$
Furthermore, there exist constants
$\epsilon_\omega\in(0,1)$ and $C_\omega>0$ such that 
$$ \omega( B_r(w)\backslash B_r(z) ) +  \omega( B_r(z)\backslash B_r(w) ) \leq C_\omega \left( { {\tt d}(w,z) \over r}  \right)^{\!\!\epsilon_\omega}   $$
for all $w,z\in bD$ such that  ${\tt d}(w,z)\leq r\leq1$.
\end{lemma}
\begin{proof}
The proof is an immediate consequence of \eqref{E:omegab}.
\end{proof}
}

\vskip0.1in

\noindent $\bullet$\quad {\em A family of Cauchy-like integrals.}\quad 
 In \cite[Sections 3 and 4]{LS2017} an ad-hoc family $\{\mathbf C_\epsilon\}_\epsilon$ of Cauchy-Fantappi\`e integrals is introduced (each determined by the aforementioned denominators $g_\epsilon(w, z)$) whose corresponding boundary operators $\{\EuScript C_\epsilon\}_\epsilon$ play a crucial role in the analysis
 of $L^p(bD, \lambda)$-regularity of the Cauchy--Szeg\H o projection. 
 We henceforth refer to 
 $\{\Cine\}_\epsilon$ as the {\em Cauchy-Leray integrals}; we record here a few relevant points for later reference.
\vskip0.1in
\begin{itemize}
 \item[{\tt [i.]}] Each $\EuScript C_\epsilon$ 
 admits a primary decomposition in terms of an ``essential part'' $\EuScript C_\epsilon^\sharp$ and a ``remainder'' $\EuScript R_\epsilon$, which are used in the proof of the $L^2(bD, \omega)$-regularity of $\EuScript C_\epsilon$. However, at this stage the magnitude of the parameter $\epsilon$ plays no role
 (this is because of the ``uniform'' estimates \eqref{gd}) and we temporarily drop reference to $\epsilon$ and simply write  $\EuScript C$ in lieu of $\EuScript C_\epsilon$; $C(w, z)$ for $C_\epsilon(w, z)$, etc.. Thus
\begin{align}\label{C operator}
   \EuScript C =  \EuScript C^\sharp+ \EuScript R,  
\end{align}
with a corresponding decomposition for the integration kernels:
\begin{equation}\label{E: C kernel}
C(w,z)=C^\sharp(w,z)+ R(w,z).
\end{equation}
\vskip0.07in

\item[]
The ``essential'' kernel $C^\sharp(w,z)$  satisfies standard size and smoothness conditions that ensure
the boundedness of $\EuScript C^\sharp$ in $L^2(bD, \omega)$ by a $T(1)$-theorem for the space of homogenous type $\{bD, {\tt d}, \omega\}$. On the other hand, the ``remainder'' kernel $R(w,z)$ satisfies improved size and smoothness conditions granting
 that the corresponding operator $\EuScript R$ is bounded in $L^2(bD, \omega)$ by elementary considerations; see \cite[Section 4]{LS2017}.
\vskip0.1in

\item[{\tt [ii.]}] 
One then turns to the Cauchy--Szeg\H o projection, for which $L^2(bD, \omega)$-regularity is trivial but $L^p(bD, \omega)$-regularity, for $p\neq 2$, is not. 
It is in this stage that the size of $\epsilon$ in the definition of the Cauchy-type boundary operators of item  {\tt [i.]} is relevant.
It turns out that each $\EuScript C_\epsilon$ admits a further, ``finer'' decomposition into (another) ``essential'' part and (another) ``reminder'', which are obtained by truncating the integration kernel $C_\epsilon (w, z)$ by a smooth cutoff function $\chi_\epsilon^s(w, z)$ that equals 1 when ${\tt d}(w, z) <s = s(\epsilon)$. One has:
\begin{equation}\label{E:Cs-op}
\EuScript C_\epsilon = \EuScript C^s_\epsilon + \EuScript R^s_\epsilon
\end{equation}
where
\begin{equation}\label{E:Ess-small}
 \|  (\EuScript C^s_\epsilon)^\dagger -  \EuScript C^s_\epsilon\|_{p} \lesssim \epsilon^{1/2} M_p
 \end{equation}
for any $1<p<\infty$, where $\displaystyle{M_p = {p\over p-1} +p}$. Here and henceforth, the upper-script ``$\dagger$'' denotes adjoint in $L^2(bD, \omega)$ (hence $(\EuScript C^s_\epsilon)^\dagger$ is the adjoint of $\EuScript C^s_\epsilon$ in $L^2(bD, \omega)$); see \cite[Proposition 18]{LS2017}.
Furthermore $\EuScript R^s_\epsilon$ and $(\EuScript R^s_\epsilon)^\dagger$ are controlled by ${\tt d}(w, z)^{-2n+1}$ and therefore are easily seen to be bounded
\begin{equation}\label{E:RsBdd}
\EuScript R^s_\epsilon,\ \ (\EuScript R^s_\epsilon)^\dagger:
L^1(bD, \omega)\to L^\infty(bD, \omega),
\end{equation}
 see \cite[(5.2) and comments thereafter]{LS2017}. 

\end{itemize}

\vskip0.1in

\noindent$\bullet$ {\em  Bounded mean oscillation on $bD$}. \quad
The space ${\rm BMO }(bD, \lambda)$ is defined as the collection of all $b\in L^1(bD, \lambda)$ such that
$$ \|b\|_*:=\sup_{ z\in bD, r>0, B_r(z)\subset bD} {1\over \lambda(B_r(z))}\int\limits_{B_r(z)} |b(w)-b_B|d\lambda(w)<\infty, $$
{with the balls $B_r(z)$ as in \eqref{ball} and where}
\begin{align}\label{fb}
b_B={1\over \lambda(B)}\int\limits_B b(z)d\lambda(z).
\end{align}
${\rm BMO }(bD, \lambda)$ is a normed space with
$\|b\|_{{\rm BMO }(bD, \lambda)}:=\|b\|_*+ \|b\|_{L^1(bD, \lambda)}. $
{We note the inclusion
\begin{equation}\label{E:BMO-Lp}
{\rm BMO }(bD, \lambda)\subset L^p(bD,  \lambda),\quad 1\leq p<\infty,
\end{equation}
which is a consequence of the John--Nirenberg inequality \cite[Corollary p. 144]{STE1} and of the compactness of $bD$.}
On account of \eqref{E:Leray Levi to sigma}, it is clear that 
$$
{\rm BMO }(bD, \sigma) = {\rm BMO }(bD, \lambda)\quad \text{with}\quad \|b\|_{{\rm BMO }(bD, \sigma)}\approx \|b\|_{{\rm BMO }(bD, \lambda)},
$$
where ${\rm BMO }(bD, \sigma)$ is the classical $BMO$ space (where the reference measure is induced Lebesgue).
\vskip0.1in

\noindent$\bullet$ {\em  Vanishing mean oscillation on $bD$}. \quad The space ${\rm VMO }(bD, \lambda)$ is the subspace of ${\rm BMO }(bD, \lambda)$ whose members satisfy the further requirement that
\begin{align}\label{vmpc}
\lim\limits_{a\rightarrow 0}\sup\limits_{B\subset bD:~r_B=a}{1\over \lambda(B)}\int\limits_{B}|f(z)-f_B|d\lambda(z)=0,
\end{align}
where $r_B$ is the radius of $B$. As before, it is clear that ${\rm VMO }(bD, \sigma) = {\rm VMO }(bD, \lambda)$.
\vskip0.1in

\noindent $\bullet$ {\em Muckenhoupt weights on $bD$.} \quad Let {$p\in(1, \infty)$. A non-negative locally integrable function $\psi$ is called an
\emph{$A_p(bD, \sigma)$-weight}, if
\begin{align*}
[\psi]_{A_p(bD, \sigma)}:=\sup_{B} \langle \psi\rangle_B\langle \psi^{1-p'}\rangle_B^{p-1}<\infty,
\end{align*}
where the supremum is taken over all balls $B$ in $bD$, and 
$ \displaystyle{\langle \phi\rangle_B:={1\over \sigma(B)}\int\limits_{B}\phi (z)\,d\sigma(z)}$. Moreover,  $\psi$ is called an
\emph{$A_1(bD, \sigma)$-weight} if $[\psi]_{A_1(bD, \sigma)}:=\inf\{C\geq0: \langle \psi\rangle_B \leq C \psi(x), \forall x\in B, \forall {\rm\ balls\ } B\in bD \} <\infty$.

\vskip0.1in
Similarly, one can define the \emph{$A_p(bD, \lambda)$-weight} for $1\leq p<\infty$.

As before, the identity \eqref{E:Leray Levi to sigma} grants that $$A_p(bD, \sigma) = A_p(bD, \lambda)\quad \text{with}\quad
[\psi]_{A_p(bD, \sigma)}\approx [\psi]_{A_p(bD, \lambda)},$$ thus we will henceforth simply write
$A_p(bD)$ and $[\psi]_{A_p(bD)}$. At times it will be more convenient to work with $A_p(bD, \lambda)$, and in this case we will refer to its members as
{\em $A_p$-like weights}.
\vskip0.1in

\noindent$\bullet$ {\em  Holomorphic Hardy spaces for Muckenhoupt weights.} Given a function $F$ holomorphic in $D$ we let $\mathcal N(F)$  denote the non-tangential maximal function of $F$, 
that is
$$ \mathcal N(F)(\xi) := \sup_{z\in \Gamma_\alpha(\xi)}|F(z)|, \quad \xi\in bD,$$
{where}
 $ \Gamma_\alpha(\xi) =\{z\in D:\ |(z-\xi)\cdot \bar{\nu}_\xi| < (1+\alpha) \delta_\xi(z), |z-\xi|^2<\alpha \delta_\xi(z)\}$, with
$\bar{\nu}_\xi$ = the (complex conjugate of) the outer unit normal vector to $\xi\in bD$, and $\delta_\xi(z)=$ the minimum between the {$($Euclidean$)$} distance of $z$ to $bD$ and the distance of $z$ to the tangent 
space  at $\xi$.
\vskip0.08in

In \cite[Proposition 1.3]{DLLW}  we have proved that the following spaces of holomorphic functions:
\begin{definition}\label{D:Hardy space}
Suppose $1\leq\, p<\infty$ and let $\Op$ be an $A_p$-measure. We define $H^p(bD, \Op)$ to be the space of functions $F$ that are holomorphic in $D$ {with}
$\mathcal N (F)\in L^p(bD, \Op)$, and set
\begin{equation}\label{E:NT-norm}
\|F\|_{H^p(bD, \Op)} := \|\mathcal N (F)\|_{L^p(bD, \Op)}
\end{equation}
\end{definition}
are closed subspaces of  $L^p(bD, \Op)$. Hence, for $p=2$ there is a (unique) orthogonal projection
$\SopO: L^2(bD, \Ot)\to H^2(bD, \Ot)$.

\section{The commutator of the Cauchy--Leray integral
}\label{S:3}
\setcounter{equation}{0}

As before, in the proofs of all statements in this section we adopt the shorthand $\Omega$ for $\Omega_p$, and $\psi$ for $\psi_p$. We begin by recalling two results from
\cite{DLLW}.
\begin{theorem}\cite{DLLW}\label{T:5.1} Let $D\subset \mathbb C^n$, $n\geq 2$, be  a bounded, strongly pseudoconvex domain of class $C^2$. Then the Cauchy-type integral $\EuScript C_\epsilon$ is bounded on $L^p(bD, \Op)$ for any $0<\epsilon<\epsilon(D)$, any $1<p<\infty$ and any $A_p$-measure $\Op$, 
  {with}
 \begin{align}\label{C quantitative bound}
 \|\EuScript C_\epsilon\|_{L^p(bD,\Op)\to L^p(bD,\Op)}\  \lesssim
 \, { c_\epsilon}\cdot [\Op]_{A_p}^{\max\{1,{1\over p-1}\}},
  \end{align}
where the implied constant depends on $p$ and $D$,  but is independent of {  $\epsilon$ or $\Op$}, { and $c_\epsilon$ is the constant in \eqref{cepsilon}}.
\end{theorem}
It follows that for any $A_2$-measure $\Ot$, the $L^2(bD, \Ot)$-adjoint
$\EuScript C_\epsilon^\spadesuit$ 
is also bounded on $L^p(bD, \Op)$ with same bound.

\begin{proposition}\cite{DLLW}\label{prop cancellation}
For any fixed $0<\epsilon<\epsilon (D)$ as in \cite{LS2017}, there exists $s=s(\epsilon)>0$ such that  
   \begin{equation}\label{E:s-to-dag-s}
  \|(\EuScript C_\epsilon^s)^\dagger-\EuScript C_\epsilon^s\|_{L^p(bD, \Op)\to L^p(bD, \Op)}\
  {\lesssim}\ \epsilon^{1/2} [\Op]_{A_p}^{\max\{1,{1\over p-1}\}}
  \end{equation}
 \noindent  $\text{for any}\ 1<p<\infty \  \text{and for any} \ A_p \text{-measure},  \Op$ where the implied constant depends on $D$ and $p$ but is independent of $\Op$ and of $\epsilon$.
  As before, here  $(\EuScript C_\epsilon^s)^\dagger$ 
   denotes the adjoint
  in $L^2(bD, \omega)$.
  \end{proposition}

\vskip0.1in
 
 Here we prove the following
 
\begin{theorem}\label{T:5.2}
Let $D\subset \mathbb C^n$, $n\geq 2$, be  a bounded, strongly pseudoconvex domain of class $C^2$ and let $\lambda$ be the Leray Levi measure for $bD$. The following hold for any $b\in L^1(bD, \lambda)$, any $1<p<\infty$ and
any $0<\epsilon< \epsilon(D)$:

\vskip0.1in
{\rm(i)} 
If $b\in{\rm BMO}(bD, \lambda)$ then
  the commutator $[b, \EuScript C_\epsilon]$ is bounded on  $L^p(bD, \Op)$ for any  $A_p$-measure $\Op$, and
$$\| [b, \EuScript C_\epsilon] \|_{L^p(bD,\Op)\to L^p(bD,\Op)} \lesssim\, \|b\|_{{\rm BMO}(bD,\lambda)} \cdot { c_\epsilon}\cdot [\Op]_{A_p}^{2\cdot\max\{1,{1\over p-1}\}}.  $$
Conversely, if $[b, \EuScript C_\epsilon]$ is bounded on  $L^p(bD, \Op)$ for some $A_p$-measure $\Op$, 
 then
$b\in{\rm BMO}(bD, \lambda)$ with 
$$ \|b\|_{{\rm BMO}(bD,\lambda)}\, {\lesssim}\, \| [b, \EuScript C_\epsilon] \|_{L^p(bD,\Op)\to L^p(bD,\Op)} .$$
The implied constants depend solely on $p$ and $D$. 
\vskip0.1in

{\rm(ii)} If $b\in{\rm VMO}(bD, \lambda)$ then
  the commutator $[b, \EuScript C_\epsilon]$ is compact on  $L^p(bD, \Op)$ for any  $A_p$-measure $\Op$. 
Conversely, if $[b, \EuScript C_\epsilon]$ is compact on  $L^p(bD, \Op)$ for some $A_p$-measure $\Op$, 
 then
$b\in{\rm VMO}(bD, \lambda)$.
\vskip0.1in

Moreover, for any $A_2$-measure $\Ot$, and with $\EuScript C_\epsilon^\spadesuit$ denoting the adjoint of 
 $\EuScript C_\epsilon$ in $L^2(bD, \Ot)$, we have that
{\rm(i)}  and {\rm(ii)}  above also hold with 
{$[b, \EuScript C_\epsilon^\spadesuit]$}
in place of $[b, \EuScript C_\epsilon]$. 
\end{theorem}
\begin{proof}

\noindent {\it Proof of Part {\rm (i)}}: We begin with proving the sufficiency. Suppose $b$ is in ${\rm BMO}(bD,\lambda)$, and we now prove that 
the commutator $[b, \EuScript C_\epsilon]$ is bounded on $L^{p}(bD, \Omega_p)$.

Note that $[b, \EuScript C_\epsilon]= [b, \EuScript C_\epsilon^\sharp]+ [b, \EuScript R_\epsilon]$.
and that $\EuScript C_\epsilon^\sharp$ is a standard Calder\'on-Zygmund operator.
Following the standard approach (see for example \cite{LOR}), 
we obtain that 
$$\| [b, \EuScript C_\epsilon^\sharp] \|_{L^{p}(bD, \Omega_p)}
\lesssim \|b\|_{{\rm BMO}(bD,\lambda)} \cdot { c_\epsilon}\cdot  [\Omega_p]_{A_p}^{2\cdot\max\{1,{1\over p-1}\}}.  $$
Thus, it suffices to verify that $ [b, \EuScript R_\epsilon]$ is bounded on $L^p(bD,\Omega_p)$ with the correct quantitative bounds.

In fact, employing the
same decomposition as in the proof of Theorem \ref{T:5.1}, we  obtain  that 
\begin{align}\label{sm}
&([b,\EuScript R_\epsilon]f)^\# (\tilde z)
\\
&\lesssim
\|b\|_{{\rm BMO} (bD,\lambda)}\left( \big(M(|\EuScript R_\epsilon f|^\alpha)(\tilde z) \big)^{1\over \alpha}+ \big(M(| f|^\beta)(\tilde z) \big)^{1\over \beta}
+\left(M(|f|^\alpha)(\tilde z)\right)^{1\over \alpha}\right),\nonumber
\end{align}
where $1<\alpha,\beta<p, \tilde z\in bD$.  Hence, we have
\begin{align}\label{[b,R] Lp}
&\left\|[b,\EuScript R_\epsilon]f\right\|_{L^p(bD,\Omega_p)}^{p}\\
&\leq C\Big(\Omega_p(bD)\Big(([b,\EuScript R_\epsilon]f)_{bD}\Big)^p+\|([b,\EuScript R_\epsilon]f)^\#\|_{L^p(bD,\Omega_p)}^{p}\Big)\nonumber\\
&\lesssim \Omega_p(bD)\Big(([b,\EuScript R_\epsilon]f)_{bD}\Big)^p+[\Op]_{A_p}^{2p\cdot \max\{1,{1\over p-1}\}}\|b\|_{{\rm BMO} (bD,\lambda)}^p\|f\|_{L^p(bD,\Omega_p)}^{p},\nonumber
\end{align}
where the second inequality follows from \eqref{sm} and Theorem \ref{T:5.1}. Now it suffice to show that
\begin{align}\label{[b,R] L1}
\Omega_p(bD)\Big(([b,\EuScript R_\epsilon]f)_{bD}\Big)^p \lesssim [\Op]_{A_p}^{p\cdot \max\{1,{1\over p-1}\}} \|b\|_{{\rm BMO} (bD,\lambda)}^p\|f\|_{L^p(bD,\Omega_p)}^{p}.
\end{align}
By H\"older's inequality and the fact that $[b,\EuScript R_\epsilon]$ is bounded on $L^q(bD,\lambda)$ for any $1<q<\infty$, see \cite{DLLWW}, we have
\begin{align*}
\Omega_p(bD)\Big(([b,\EuScript R_\epsilon]f)_{bD}\Big)^p
&\leq  \Omega_p(bD) \left({1\over \lambda(bD)}\int\limits_{bD} \big|[b,\EuScript R_\epsilon]f(z)\big|^q  d\lambda(z)\right)^{p\over q}\\
&\lesssim \|b\|_{{\rm BMO} (bD,\lambda)}^p\Omega_p(bD) \left({1\over \lambda(bD)}\int\limits_{bD} \big|f(z)\big|^q  d\lambda(z)\right)^{p\over q}\\
&\lesssim  \|b\|_{{\rm BMO} (bD,\lambda)}^p\int\limits_{bD}\Big(M(|f|^q)(z)\Big)^{p\over q}\psi(z)d\lambda(z)\\
&\lesssim [\Op]_{A_p}^{p\cdot \max\{1,{1\over p-1}\}} \|b\|_{{\rm BMO} (bD,\lambda)}^p\|f\|_{L^p(bD,\Omega_p)}^{p}.
\end{align*}
Therefore, \eqref{[b,R] L1} holds, which, together with \eqref{[b,R] Lp}, implies that
$[b,\EuScript R_\epsilon]$ is bounded on $L^p(bD,\Omega_p).$   Combining the estimates for $[b, \EuScript C_\epsilon^\sharp]$ and $ [b, \EuScript R_\epsilon]$, we obtain that 
$$\| [b, \EuScript C_\epsilon] \|_{L^{p}(bD, \Omega_p)}
 \lesssim \|b\|_{{\rm BMO}(bD,\lambda)} \cdot { c_\epsilon}\cdot  [\Op]_{A_p}^{2\cdot\max\{1,{1\over p-1}\}}.  $$

We now prove the necessity. Suppose $b$ is in $L^1(bD,\lambda)$ and $[b,\EuScript C_\epsilon]$ is bounded on $L^p(bD,\Omega_p)$ for some $1<p<\infty$.

Let   $ C^\sharp_{1,\epsilon} (z,w)$ and $ C^\sharp_{2,\epsilon}(z,w)$ be
the real and imaginary parts of $ {C^\sharp_\epsilon(z,w)}$, respectively.
And let $ R_{1,\epsilon}(z,w)$ and $ R_{2,\epsilon}(z,w)$ be
the real and imaginary parts of $ {R_\epsilon(z,w)}$, respectively.  Then, combining the size and smoothness conditions \cite[(3.1), (3.2) in Theorem 3.1]{DLLW},
we get that there exist positive constants $\gamma_0$, $\mathcal A_3$, $\mathcal A_4$ and $\mathcal A_5$ such that 
for every ball $B=B_r(z_0)\subset bD$ with $r<\gamma_0$, there exists another ball
$\tilde B = B_r(w_0)\subset bD$ with  $\mathcal A_3r \leq {\tt d}(w_0,z_0) \leq (\mathcal A_3+1)r$
such that
at least one of the following properties holds:

$\mathfrak a)$ For every $z\in B$ and $w\in \tilde B$, $ C^\sharp_1(w,z)$ does not change sign and $| C^\sharp_{1,\epsilon}(z,w)|\geq {\mathcal  A_4\over {\tt d}(w,z)^{2n} };$

$\mathfrak b)$ For every $z\in B$ and $w\in \tilde B$, $ C^\sharp_2(w,z)$ does not change sign and $| C^\sharp_{2,\epsilon}(z,w)|\geq {\mathcal  A_5\over {\tt d}(w,z)^{2n} }.$

Then, without loss of generality, we assume that the property $\mathfrak a)$ holds. Then combining with the size estimate of $ {R(z,w)}$ as in the proof of \cite[Theorem 3.1]{DLLW} (see (3.4) there),
we obtain that there exists a positive constant $\mathcal A_6$ such that for every $z\in B$ and $w\in \tilde B$, $ C^\sharp_1(w,z)+ R_1(w,z)$
does not change sign and that 
\begin{align}\label{kernel lower bound C1 dag}
| C^\sharp_{1,\epsilon}(w,z)+ R_{1,\epsilon}(w,z)|\geq {\mathcal  A_6 \over  {\tt d}(w,z)^{2n} }.
\end{align}

We test the ${\rm BMO}(bD, \lambda)$ condition on the case of balls with big radius and small radius.
Case 1: In this case we work with balls with a large radius, $r\geq \gamma_0$. 

By \eqref{E:omegab} and by the fact that 
$\lambda(B)\geq \lambda( B_{\gamma_0}(z_0)) \approx \gamma_0^{2n}$, we obtain that
\begin{align*}
{1\over \lambda(B)} \int\limits_B|b(z)-b_B| d\lambda(z) & \lesssim  {1\over \lambda( B_{\gamma_0}(z_0))}  \|b\|_{L^1(bD, \lambda)} 
\lesssim \gamma_0^{-2n} \|b\|_{L^1(bD, \lambda)}.
\end{align*}
Case 2: In this case we work with balls with a small radius, $0<r<\gamma_0$.

We aim to prove that for every fixed ball $B=B_r(z_0)\subset bD$ with radius $r<\gamma_0$,
\begin{align}\label{BMO bd by C dag}
{1\over \lambda(B)} \int\limits_B|b(z)-b_B|d\lambda(z) \lesssim \|[b, \EuScript C_\epsilon] \|_{ L^{p}(bD, \Omega_p)\to L^{p}(bD, \Omega_p)},
\end{align}
which, combining with Case 1, finishes the proof of the necessity part.

Now let $\tilde B=B_r(w_0)$ be the ball chosen as above, and let $m_b(\tilde B)$ be the median value of $b$ on the ball $\tilde B$ with respect to the measure $\lambda$ defined as follows:
$m_b(\tilde B)$ is a real number that satisfies simultaneously
$$ \lambda(\{w\in \tilde B: b(w)>m_b(\tilde B)\})\leq {1\over2}\lambda(\tilde B)\quad {\rm and}\quad \lambda(\{w\in \tilde B: b(w)<m_b(\tilde B)\}) \leq {1\over2}\lambda(\tilde B).$$
Then, following the idea in \cite[Proposition 3.1]{LOR} 
by the definition of median value, we choose
$F_1:=\{ w\in \tilde B: b(w)\leq m_b(\tilde B) \}$ and
$F_2:=\{ w\in \tilde B: b(w)\geq m_b(\tilde B) \}$. Then it is direct that $\tilde B=F_1\cup F_2$, and moreover, from the definition of $m_b(\tilde B)$, we see that
\begin{align}\label{f1f2-1 dag}
\lambda(F_i) \geq{1\over 2} \lambda(\tilde B),\quad i=1,2.
\end{align}

Next we define
$E_1=\{ z\in B: b(z)\geq m_b(\tilde B) \}$ and  $
E_2=\{ z\in B: b(z)< m_b(\tilde B)\} $.
Then $B=E_1\cup E_2$ and $E_1\cap E_2=\emptyset$. 
Then it is clear that $b(z)-b(w)$ is non-negative for any $(z,w)\in E_1\times F_1$, and is negative for any $(z,w)\in E_2\times F_2$.
Moreover, for $(z,w)$ in $(E_1\times F_1 )\cup (E_2\times F_2)$, we have 
\begin{align}\label{bx-by-1 dag}
|b(z)-b(w)| 
&\geq |b(z)-m_b(\tilde B)|. 
\end{align}

Therefore, from \eqref{kernel lower bound C1 dag}, \eqref{f1f2-1 dag}, and \eqref{bx-by-1 dag} 
we obtain that 
\begin{align}
&{1\over \lambda(B)}\int\limits_{E_{1}}\big|b(z)-m_b(\tilde B)\big|d\lambda(z)\nonumber\\
&\lesssim 
{1\over \lambda(B)}{\lambda{(F_1)}\over  \lambda(B)}\int\limits_{E_{1}}\big|b(z)-m_b(\tilde B)\big| d\lambda(z) 
\nonumber\\ 
&\lesssim
{1\over \lambda(B)}\int\limits_{E_{1}}\int\limits_{F_{1}}{1\over {\tt d}(w,z)^{2n}}\big|b(z)-b(w)\big|  d\lambda(w) d\lambda(z)
\nonumber\\
&\lesssim
{1\over \lambda(B)}\int\limits_{E_{1}}\int\limits_{F_{1}}| { C^\sharp_{\epsilon,1}(w,z) }+  {R_{1,\epsilon}(w,z)}|\big(b(z)-b(w)\big) d\lambda(w)d\lambda(z)
\nonumber\\
&\lesssim
{1\over \lambda(B)}\int\limits_{E_{1}}\bigg|\int\limits_{F_{1}}{C_\epsilon(w,z)} \big(b(z)-b(w)\big)d\lambda(w)\bigg|d\lambda(z)\nonumber\\
&\lesssim
{1\over \lambda(B)}\int\limits_{E_{1}}\left|[b, \EuScript C_\epsilon](\chi_{F_1})(z)\right| 
d\lambda(z),
\label{key 1 dag}
\end{align}
where the last but second inequality follows from the fact that ${ C^\sharp_{\epsilon,1}(w,z) }+  {R_{1,\epsilon}(w,z)}$ is the real part of 
${C_\epsilon(w,z)} $.

Then, by using H\"older's inequality and the condition that $\Op\in A_p$ with the density function $\psi$, we further obtain that the right-hand side of \eqref{key 1 dag} is bounded by
\begin{align*}
&
{1\over \lambda(B)} \left(\,\int\limits_{E_1} \psi^{-{ p'\over p }}(z) d\lambda(z) \right)^{1\over p'}\bigg(\,\int\limits_{E_{1}}\left|[b, \EuScript C_\epsilon](\chi_{F_1})(z)\right|^{p } \psi(z)d\lambda(z)\bigg)^{1\over p}\\
 &\lesssim
  {1\over \Omega_p(B)}\lambda(B)\left(\Omega_p(B)\right)^{-{1\over p}}\left(\Omega_p(F_1)\right)^{1\over p}
\|[b, \EuScript C_\epsilon]\|_{L^{p}(bD, \Omega_p)\to L^{p}(bD, \Omega_p)}\\ 
 &\lesssim
\|[b, \EuScript C_\epsilon]\|_{L^{p}(bD, \Omega_p)\to L^{p}(bD, \Omega_p)}.
\end{align*}
Similarly, we can obtain that 
\begin{align*}
{1\over \lambda(B)}\int\limits_{E_{2}}\big|b(z)-m_b(\tilde B)\big|d\lambda(z)
 &\lesssim
\|[b, \EuScript C_\epsilon]\|_{L^{p}(bD, \Omega_p)\to L^{p}(bD, \Omega_p)}.
\end{align*}

As a consequence, we get that
\begin{align*}
&{1\over \lambda(B)}\int\limits_{B}\big|b(z)-m_b(\tilde B)\big|d\lambda(z)\\
&\lesssim{1\over \lambda(B)}\int\limits_{E_1}\big|b(z)-m_b(\tilde B)\big|d\lambda(z)
+{1\over \lambda(B)}\int\limits_{E_2}\big|b(z)-m_b(\tilde B)\big|d\lambda(z)\\
&\lesssim
\|[b, \EuScript C_\epsilon]\|_{L^{p}(bD, \Omega_p)\to L^{p}(bD, \Omega_p)}.
\end{align*}

Therefore,
\begin{align*}
{1\over \lambda(B)}\int\limits_{B}\big|b(z)-b_B\big|d\lambda(z)
&\leq {2\over \lambda(B)}\int\limits_{B}\big|b(z)-m_b(\tilde B)\big|d\lambda(z)
\lesssim \|[b, \EuScript C_\epsilon]\|_{L^{p}(bD, \Omega_p)\to L^{p}(bD, \Omega_p)},
\end{align*}
which gives 
\eqref{BMO bd by C dag}. 
Combining the estimates in Case 1 and Case 2 above, we see that 
$b$ is in BMO$(bD,\lambda)$.  The proof of Part (1) is concluded.

\vskip0.1in
\noindent {\it Proof of Part {\rm (ii)}}: 
We begin by showing the sufficiency. Suppose $b\in {\rm VMO}(bD,\lambda)$.
Note that $[b, \EuScript C_\epsilon]= [b, \EuScript C_\epsilon^\sharp]+ [b, \EuScript R_\epsilon]$, and that 
$[b, \EuScript C_\epsilon^\sharp]$ is a compact operator on $L^p(bD,\omega)$ (following the standard argument in \cite{KL2}), it suffices to verify that $ [b, \EuScript R_\epsilon]$ is compact on $L^p(bD,\Omega_p)$.
However, this follows from the approach in the proof of (ii)
of Theorem D (\cite[Theorem 1.1]{DLLWW}).

\medskip

We now prove the necessity. Suppose that $b\in {\rm BMO}(bD,\lambda)$ and that $[b, \EuScript C_\epsilon]$ is compact on $L^p(bD, \Omega)$ for some $1<p<\infty$. Without loss of generality, we assume that 
$\|b\|_{{\rm BMO}(bD,\lambda)}=1$.
\smallskip

We now use the idea from \cite{DLLWW}.
To show $b\in {\rm VMO}(bD, \lambda)$, we  seek the contradiction:  there is no bounded operator $T \;:\; \ell^{p} (\mathbb N) \to \ell^{p} (\mathbb N)$ with $Te_{j } = T e_{k} \neq 0$ for all $j,k\in \mathbb N$.  Here, $e_{j}$ is the 
standard basis for $\ell^{p} (\mathbb N)$.  
Thus, it suffices to construct the approximates to a standard basis in $\ell^{p}$, namely a sequence of functions $\{g_{j}\}$ such that 
$ \lVert g_{j } \rVert _{L^{p}(bD, \Omega_p) } \simeq 1$,  and for a nonzero $ \phi$,   we have $\lVert \phi - [b, \EuScript C_\epsilon] g_{j}\rVert 
 _{L^{p}(bD, \Omega_p) } <2^{-j}$.  
 
Suppose that $b\notin {\rm VMO}(bD, \lambda)$, then there exist $\delta_0>0$ and a sequence $\{B_j\}_{j=1}^\infty:=\{B_{r_j}(z_j)\}_{j=1}^\infty$ of balls such that 
\begin{align}\label{delta0}
{1\over \lambda(B_j)} \int\limits_{B_j} |b(z)-b_{B_j}| d\lambda(z)\geq \delta_0. 
\end{align}

Without loss of generality, we assume that for all $j$, $r_j<\gamma_0$, where $\gamma_0$ is 
the fixed constant in the argument for \eqref{kernel lower bound C1 dag}.

Now choose a subsequence $\{B_{j_i}\}$ of $\{B_j\}$ such that
\begin{align}\label{ratio1}
r_{j_{i+1}} \leq{1\over 4c_\omega} r_{j_{i}},
\end{align}
where $c_\omega$ is the constant such that
\begin{align}
c_\omega^{-1} r^{2n}\leq \lambda\big(B_r(w) \big)\leq c_\omega r^{2n},\quad 0<r\leq 1.
\end{align}
For the sake of simplicity we drop the subscript $i$, i.e., we still denote $\{B_{j_i}\}$ by $\{B_{j}\}$.

Then for each such $B_j$, we can choose a corresponding $\tilde B_j$.
Now let $m_b(\tilde B_j)$ be the median value of $b$ on the ball $\tilde B_j$ with respect to the measure $\omega d\sigma$.
Then,  by the definition of median value, we can
find disjoint subsets $F_{j,1},F_{j,2}\subset \tilde B_j$ such that
\begin{align*}
F_{j,1}\subset\{ w\in \tilde B_j: b(w)\leq m_b(\tilde B_j) \},\quad
 F_{j,2}\subset\{ w\in \tilde B_j: b(w)\geq m_b(\tilde B_j) \},
\end{align*}
and  
\begin{align}\label{f1f2-1 1}
\lambda(F_{j,1}) = \lambda(F_{j,2}) ={\lambda(\tilde B_j)\over 2} .
\end{align}

Next we define
$ E_{j,1}=\{ z\in B: b(z)\geq m_b(\tilde B_j) \},\ \ 
E_{j,2}=\{ z\in B: b(z)< m_b(\tilde B_j) \},
$
then $B_j=E_{j,1}\cup E_{j,2}$ and $E_{j,1}\cap E_{j,2}=\emptyset$. 
Then it is clear that 
$b(z)-b(w) \geq 0$ for $ (z,w)\in E_{j,1}\times F_{j,1}$ and 
$b(z)-b(w) < 0$ for $ (z,w)\in E_{j,2}\times F_{j,2}$.
And for $(z,w)$ in $(E_{j,1}\times F_{j,1} )\cup (E_{j,2}\times F_{j,2})$, we have 
\begin{align}\label{bx-by-1 1}
|b(z)-b(w)| 
&\geq |b(z)-m_b(\tilde B_j)|. 
\end{align}

We now consider 
$$\widetilde F_{j,1}:= F_{j,1}\bigg\backslash \bigcup_{\ell=j+1}^\infty \tilde B_\ell\quad{\rm and}\quad \widetilde F_{j,2}:= F_{j,2}\bigg\backslash \bigcup_{\ell=j+1}^\infty \tilde B_\ell,\quad {\rm for}\ j=1,2,\ldots.$$
Then, based on the decay condition of the radius $\{r_j\}$, we obtain that
for each $j$,
\begin{align}\label{Fj1}
\lambda(\widetilde F_{j,1}) &\geq \lambda(F_{j,1})- \lambda\Big( \bigcup_{\ell=j+1}^\infty \tilde B_\ell\Big) \geq
{1\over 2} \lambda(\tilde B_j)-\sum_{\ell=j+1}^\infty \lambda\big(  \tilde B_\ell\big)\nonumber\\
&\geq {1\over 2} \lambda(\tilde B_j)- {c_\lambda^2\over (4c_\lambda)^{2n}-1}\lambda(\tilde B_j)\geq {1\over 4} \lambda(\tilde B_j).
\end{align}

Now for each $j$, we have that
\begin{align*}
&{1\over \lambda(B_j)} \int\limits_{B_j} |b(z)-b_{B_j}| d\lambda(z)\\
&\leq{2\over \lambda(B_j)}\int\limits_{B_{j}}\big|b(z)-m_b(\tilde B_j)\big|d\lambda(z)\\
&= {2\over \lambda(B_j)}\int\limits_{E_{j,1}}\big|b(z)-m_b(\tilde B_j)\big|d\lambda(z) + {2\over \lambda(B_j)}\int\limits_{E_{j,2}}\big|b(z)-m_b(\tilde B_j)\big|d\lambda(z).
\end{align*}
Thus, combining with \eqref{delta0} and the above inequalities, we obtain that as least one of the following inequalities holds:
\begin{align*}
{2\over \lambda(B_j)}\int\limits_{E_{j,1}}\big|b(z)-m_b(\tilde B_j)\big|d\lambda(z) \geq {\delta_0\over2},\quad 
{2\over \lambda(B_j)}\int\limits_{E_{j,2}}\big|b(z)-m_b(\tilde B_j)\big|d\lambda(z) \geq {\delta_0\over2}.
\end{align*}
We may assume that the first one holds, i.e., 
\begin{align*}
{2\over \lambda(B_j)}\int\limits_{E_{j,1}}\big|b(z)-m_b(\tilde B_j)\big| d\lambda(z) \geq {\delta_0\over2}.
\end{align*}
Therefore, for each $j$, from \eqref{f1f2-1 1} and \eqref{bx-by-1 1} and by using \eqref{key 1 dag}, we obtain that 
\begin{align*}
{\delta_0\over4}&\leq{1\over \lambda(B_j)}\int\limits_{E_{j,1}}\big|b(z)-m_b(\tilde B_j)\big|d\lambda(z)\\
&\lesssim
{1\over \lambda(B_j) } \left(\int\limits_{E_{j,1}}\psi^{-{p'\over p}}(z)d\lambda(z)\right)^{1\over p'}\bigg( \int\limits_{bD}\big|[b, \EuScript C_\epsilon](\chi_{\widetilde F_{j,1}})(z)\big|^p\psi(z)d\lambda(z) \bigg)^{1\over p}\\
&\lesssim
{1\over \lambda(B_j) }\lambda(B_j) \Omega_p(B_j)^{-{1\over p}}\bigg( \int\limits_{bD}\big|[b, \EuScript C_\epsilon](\chi_{\widetilde F_{j,1}})(z)\big|^p\psi(z)d\lambda(z) \bigg)^{1\over p}\\
&\lesssim
\bigg( \int\limits_{bD}\big|[b, \EuScript C_\epsilon](f_j)(z)\big|^p\psi(z)d\lambda(z) \bigg)^{1\over p},
\end{align*}
where $ f_j := {\chi_{\widetilde F_{j,1}} \over \Omega_p(B_j)^{1\over p}}.$
Combining the above estimates we obtain  
$$ 0<\delta_0  \lesssim
\bigg( \int\limits_{bD}\big|[b, \EuScript C_\epsilon](f_j)(z)\big|^{p}\psi(z)d\lambda(z) \bigg)^{1\over p}.
$$

Moreover, since
 $\psi\in A_p$, it follows that there exist positive constants $C_1,C_2$ and $\sigma\in (0,1)$ such that for any measurable set $E\subset B$,
$$\Big({\lambda(E)\over \lambda(B)}\Big)^p\leq C_1{\Omega_p(E)\over \Omega_p(B)}\leq C_2\Big({\lambda(E)\over \lambda(B)}\Big)^{\sigma}.$$
Hence, from \eqref{Fj1}, we obtain that
$ {4^{-{1\over p}}}\lesssim \|f_j\|_{L^{p}(bD,\Omega_p )}\lesssim 1.  $
Thus, it is direct to see that $\{f_j\}_j$ is a bounded sequence in $L^{p}(bD,\Omega_p )$ with a uniform $L^{p}(bD,\Omega_p)$-lower bound away from zero.

Since $[b, \EuScript C_\epsilon]$ is compact, we obtain that the  sequence
$ \{[b, \EuScript C_\epsilon](f_j)\}_j $
has a convergent subsequence, denoted by
$$ \{[b, \EuScript C_\epsilon](f_{j_i})\}_{j_i}. $$
We denote the limit function by $g_0$, i.e.,
$$ [b, \EuScript C_\epsilon](f_{j_i})\to g_0 \quad{\rm in\ }  L^{p}(bD,\Omega_p ), \quad{\rm as\ } i\to\infty.$$
Moreover,  $g_{0} \neq 0$. 

After taking a further subsequence, labeled $\{g_{j}\}_{j=1}^\infty$, we have 
\begin{itemize} 
\item  $\lVert g_{j} \rVert _{L^{p} (bD, \Omega_p  )}  \simeq 1$; 
\item  $g_{j}$ are disjointly supported;
\item and $\lVert g_{0}  -  [b, \EuScript C_\epsilon] g_{j} \rVert _{L^{p} (bD, \Omega_p  )}  < 2^{-j}$.  
\end{itemize}

Take $a_{j } = j ^{ -1}$, so that $ \{a_{j}\}_{j=1}^\infty \in \ell^{p} \setminus \ell^{1}$.   It is immediate that $ \gamma = \sum_{j} a_{j} g_{j} \in 
L^{p} (bD, \Omega_p )$, hence $ [b, \EuScript C_\epsilon] \gamma \in L^{p} (bD, \Omega_p  )$.  But, $ g_{0} \sum_{j} a_{j} \equiv \infty$, and yet 
 \begin{align*}
\Bigl\lVert  g_{0} \sum_{j} a_{j} \Bigr\rVert _{L^{p} (bD, \Omega_p  )} 
& \leq \lVert  [b, \EuScript C_\epsilon] \gamma  \rVert _{L^{p} (bD, \Omega_p  )}  
+ \sum_{j=1}^\infty a_{j}  \lVert   g_{0} -  [b, \EuScript C_\epsilon] g_{j}  \rVert _{L^{p} (bD, \Omega_p  )}  < \infty.
\end{align*}
This contradiction shows that $b\in {\rm VMO}(bD, \lambda )$.

 
Note that all the functions $f_j$ are pairwise disjointly supported. We then take non-negative numerical sequence $\{a_j\}$ with 
$$ \|\{a_i\}\|_{\ell^p}<\infty \quad {\rm but}\quad   \|\{a_i\}\|_{\ell^1}=\infty.  $$
Then there holds
\begin{align*}
&\sum_{i=1}^\infty\bigg( a_i\|f_0\|_{L^p(bD,\omega d\sigma)} -  a_i \|f_0- [b, \EuScript C](f_{j_i})\|_{L^p(bD, \lambda)} \bigg)\\
&\leq \bigg\| \sum_{i=1}^\infty a_i [b, \EuScript C](f_{j_i}) \bigg\|_{L^p(bD, \lambda)}=
 \bigg\|  [b, \EuScript C]\Big(\sum_{i=1}^\infty a_i f_{j_i}\Big) \bigg\|_{L^p(bD, \lambda)}\lesssim  
\bigg\| \sum_{i=1}^\infty a_i f_{j_i} \bigg\|_{L^p(bD, \lambda)}\\
&\lesssim \|\{a_i\}\|_{\ell^p}.
\end{align*}
Above, we use the triangle inequality, and then the upper bound on the norm of the commutator, and then the disjoint support condition. But the left-hand side is infinite by design because
$$  \sum_{i=1}^\infty  a_i\|f_0\|_{L^p(bD, \lambda)}\gtrsim  \sum_{i=1}^\infty  a_i\delta_0=+\infty$$
and
\begin{align*}
\sum_{i=1}^\infty  a_i \|f_0- [b, \EuScript C_\epsilon](f_{j_i})\|_{L^p(bD, \lambda)} &\leq \bigg(\sum_{i=1}^\infty  a_i^p \bigg)^{1\over p}
\bigg(\sum_{i=1}^\infty \|f_0- [b, \EuScript C_\epsilon](f_{j_i})\|_{L^p(bD, \lambda)}^{p'}\bigg)^{1\over p'}\\
&\leq \|\{a_i\}\|_{\ell^p}\bigg(\sum_{i=1}^\infty 2^{-ip'}\bigg)^{1\over p'}\\
&\lesssim \|\{a_i\}\|_{\ell^p},
\end{align*}
which is a contradiction.  The proof of Part (2) is concluded, completing 
 the proof of Theorem \ref{T:5.2}.
\end{proof}

 \begin{remark}\label{remark power} 
 {  We point out that the term $\epsilon^{1/2}$ can be improved to $\epsilon^{\delta}$ for any fixed small $\delta>0$, according to \cite[Remark D]{LS2017} via choosing $\beta$ there arbitrarily close to 1.}  
  \end{remark}

\vskip0.2in

\section{
 The commutator of $\Sopo$ in $L^p(bD, \Op)$}
 \label{S:4}
\setcounter{equation}{0}

\subsection{A preliminary result}
Before tackling the commutator of $\Sopo$ in the maximal class $L^p(bD, \Op)$ we need to study its behavior on its subclass $L^p(bD, \omega)$ (that is, for Leray Levi measures; recall that Leray Levi-like measures are $A_p(bD)$-measures for any $1<p<\infty$)).

\begin{theorem}\label{cauchy1}
Let $D\subset \mathbb C^n$, $n\geq 2$, be  a bounded, strongly pseudoconvex domain of class $C^2$ and let $\lambda$ be the Leray Levi measure for $bD$. 
\noindent  The following hold for any $b\in L^2(bD, \lambda)$ and any $1<p<\infty$:
\vskip0.1in
$(1)$  If $b\in{\rm BMO}(bD, \lambda)$ then the commutator $[b, \Sopo]$ is bounded on  $L^p(bD, \omega)$ for any Leray Levi-like measure $\omega$ with 
$$ \|[b, \Sopo]\|_{L^p(bD, \omega)\to L^p(bD, \omega)} \lesssim \|b\|_{{\rm BMO}(bD, \lambda)};$$
Conversely,  {  suppose that both $[b, \Sopo]$ and $[b,\Cine](I-\Sopo)$ are  bounded on  $L^p(bD, \omega)$}  for some Leray Levi-like measure $\omega$,
then $b\in{\rm BMO}(bD, \lambda)$ with 
 \begin{align*}
\|b\|_{{\rm BMO}(bD, \lambda)}&\lesssim (1+ \|\Cine\|_{L^p(bD, \omega)\to L^p(bD, \omega) })
\|[b,\Sopo]\|_{L^p(bD, \omega)\to L^p(bD, \omega) }\nonumber\\
&\qquad+
\|[b,\Cine](I-\Sopo)\|_{L^p(bD, \omega)\to L^p(bD, \omega) }.
 \end{align*}
Here the implicit constants depend only on $p$, $D$ and $\omega$.

\vskip0.1in
$(2)$  If $b\in{\rm VMO}(bD, \lambda)$ then the commutator $[b, \Sopo]$ is compact on  $L^p(bD, \omega)$ for any Leray Levi-like measure $\omega$. 
Conversely, if both $[b, \Sopo]$ and $[b,\Cine](I-\Sopo)$ are compact on $L^p(bD, \omega)$  for some Leray Levi-like measure $\omega$,
then $b\in{\rm VMO}(bD, \lambda)$.
\vskip0.1in

\noindent The implied constants in $(1)$ and $(2)$ depend solely on $     p$, $\omega$ and $D$.
\end{theorem}

  \begin{proof}[Proof of Part \textnormal{(1)}]
We first prove the sufficiency: we suppose that $b \in {\rm BMO}(bD, \lambda)$ and show that
$[b, \Sopo]: L^p(bD, \omega) \to L^p(bD, \omega)$ is bounded for all $1<p<\infty$.
Note that by duality it suffices to show that
 $[b, \Sopo]: L^p(bD, \omega) \to L^p(bD, \omega)$ is bounded for $1<p\leq 2$. 
 
We first establish boundedness in $L^2(bD, \omega)$. 
The starting point are the  following basic identities for any fixed $0<\epsilon<\epsilon(D)$:   
  \begin{equation}\label{E:basic}
  \Sopo \EuScript C_\epsilon^\dagger f =
  (\EuScript C_\epsilon \EuScript S_\lambda^\dagger)^\dagger f = (\EuScript C_\epsilon \Sopo)^\dagger f = (\Sopo)^\dagger f = \Sopo f\, ,
  \end{equation}
which are valid for any $f\in L^2(bD, \omega)$ and for any $\epsilon$ (whose value is of no import here). We recall that the upper-script ``$\dagger$'' denotes the adjoint in $L^2(bD, \omega)$.

A computation that uses \eqref{E:basic} gives that
\begin{equation}\label{E:2}
-\Sopo [b, T_\epsilon]f + 
 \Sopo b T_\epsilon f = \EuScript C_\epsilon( bf)
\end{equation}
is true with
$$
T_\epsilon : = I-(\EuScript C^\dagger_\epsilon - \EuScript C_\epsilon)
$$
whenever $f$ is taken in the H\"older-like subspace \eqref{E:Holder} --  the latter ensuring that 
all terms in \eqref{E:2} are meaningful; more precisely for such functions $f$ we have that $bf\in L^2(bD, \omega)$, since
$b\in {\rm BMO}(bD, \lambda) \subset L^2(bD, \lambda)$ on account of \eqref{E:BMO-Lp}, and
$L^2(bD, \lambda) = L^2(bD, \omega)$ by \eqref{E:Leray Levi to sigma}. We also have that  $bT_\epsilon f\in L^2(bD, \omega)$ because $T_\epsilon f \in C(bD)$ by \cite[Proposition 6 and (4.1)]{LS2017}. On the other hand, the
classical Kerzman--Stein identity  \cite{KS}
\begin{equation}\label{E:KS}
  \Sopo T_\epsilon f
  =\EuScript C_\epsilon f,\quad f\in L^2(bD, \omega),
 \end{equation}
 gives that
  \begin{equation}\label{E:1}
  b  \Sopo T_\epsilon f= b \EuScript C_\epsilon f,\quad f\in L^2(bD, \omega).
  \end{equation}
 Combining \eqref{E:2} and \eqref{E:1}  we obtain
 \begin{equation}\label{E:dense1}
 [ b, \Sopo ] T_\epsilon f = \big([b , \EuScript C_\epsilon ]\, -\, \Sopo  [b, T_\epsilon ]\big)f 
 \end{equation}
 whenever $f$ is in the H\"older-like space \eqref{E:Holder}. However the righthand side of \eqref{E:dense1}
 is meaningful and indeed bounded in $L^2(bD, \omega)$ by Theorem \ref{T:5.2} (which applies to Leray Levi-like measures);
  thus \eqref{E:dense1}
 extends to an identity on $L^2(bD, \omega)$.
Furthermore, we have  that $T_\epsilon$ is invertible in $L^2(bD, \omega)$ as a consequence of
the following two facts (1.), $\EuScript C_\epsilon$ and $(\EuScript C_\epsilon)^\dagger$ are bounded in $L^2(bD, \omega)$ and (2.),
 $T_\epsilon$ is skew adjoint (that is, $(T_\epsilon)^\dagger = - T_\epsilon$); see  the proof in \cite[p. 68]{LS2004} which applies verbatim here. We conclude that
 \begin{equation}\label{E:dense-2}
 [ b, \Sopo ] g = \big([b , \EuScript C_\epsilon ]\, -\, \Sopo  [b, T_\epsilon ]\big)\circ T_\epsilon^{-1} g,\quad g\in L^2(bD, \omega).
 \end{equation}
But the righthand side of \eqref{E:dense-2} is bounded in $L^2(bD, \omega)$ by what has just been said. Thus $[b, \Sopo]$ is also bounded, with 
\begin{align}\label{bS L2norm}  
\|[b, \Sopo]\|_{2}\lesssim \|T_\epsilon ^{-1}\|_{2}\,\|b\|_{{\rm BMO}(bD, \lambda)}\lesssim\|b\|_{{\rm BMO}(bD, \lambda)}.
\end{align}

\vskip0.1in

We next prove boundedness on $L^p(bD, \omega)$ for $1<p<2$ (as we will see in \eqref{E:choose-eps} below, it is at this stage that the choice of $\epsilon$ is relevant).
We start by combining the ``finer'' decomposition of $\EuScript C_\epsilon$, see \eqref{E:Cs-op}, 
with the classical Kerzman--Stein identity \eqref{E:KS}, which give us
\begin{equation}\label{E:new-dec}
\EuScript C_\epsilon 
= \Sopo\big(\mathcal T_\epsilon^s + \mathcal R_\epsilon^s\big)\quad \mbox{in}\ 
L^2(bD, \omega),  
\end{equation}
where
$$ \mathcal T_\epsilon^s
 := I- \big((\EuScript C_\epsilon^s)^\dagger - \EuScript C_\epsilon^s\big)\ \equiv\ I-\EuScript E_\epsilon^s $$
see \eqref{E:Ess-small}, and
$$\mathcal R_\epsilon^s := \EuScript R_\epsilon^s - (\EuScript R_\epsilon^s)^\dagger$$
see \eqref{E:RsBdd}. Plugging \eqref{E:new-dec} in \eqref{E:dense1} gives us
\begin{equation}\label{E:dense3}
 [ b, \Sopo ] \mathcal T_\epsilon^s f = \big([b , \EuScript C_\epsilon]\, -\, \Sopo  [b, T_\epsilon ] \, -\, [b, \Sopo]\mathcal R_\epsilon^s\big)f 
 \end{equation}
 whenever $f$ is in the H\"older-like space \eqref{E:Holder}. We claim that all three terms in the righthand side of \eqref{E:dense3} are in fact meaningful in $L^p(bD, \omega)$: the first two terms are so by the results of \cite{DLLWW} and \cite{LS2017}; on the other hand, the boundedness of the third term is a consequence of the boundedness of $[b,  \Sopo]$ in $L^2(bD, \omega)$ that was just proved, and of the mapping properties \eqref{E:RsBdd}, giving us:
 \begin{align}\label{E:mapping-1}
[b,\Sopo] \mathcal R_\epsilon^s : 
L^p(bD, \omega) &\hookrightarrow L^1(bD, \omega) \to L^\infty(bD, \omega) \\\notag
&\hookrightarrow L^2(bD, \omega) \to L^2(bD, \omega) \hookrightarrow L^p(bD, \omega).
\end{align}
It is at this point that it is necessary to make a specific choice of $\epsilon$.  Given $1<p<2$ we pick
 $\epsilon$ (hence $s= s(\epsilon)$) sufficiently small so that the  operator $\mathcal T_\epsilon^s$ is invertible on $L^p(bD, \omega)$ (with bounded inverse) on account of \eqref{E:Ess-small}. That is:
\begin{equation}\label{E:choose-eps}
\epsilon^{1/2} M_p:= \epsilon^{1/2}\left({p\over p-1} +p\right) <1\, .
\end{equation}
Combining \eqref{E:dense3} with the above considerations we obtain
\begin{equation}\label{E:dense-p}
 [ b, \Sopo ] g = \big([b , \EuScript C_\epsilon]\, -\, \Sopo  [b, T_\epsilon ] \, -\, [b, \Sopo]\mathcal R_\epsilon^s\big)\circ ( \mathcal T_\epsilon^s)^{-1} g,\quad g\in L^p(bD, \omega).
 \end{equation}
We conclude that $ [b,\Sopo]$ is bounded on $L^p(bD, \omega)$ with
$$ 
\|[b, \Sopo]\|_{p}\lesssim
 \bigg(1+ \|\Sopo\|_{p}+\,
 \|T_\epsilon^{-1}\|_{2}
 \|\mathcal R_\epsilon^s\|_{1,\infty}
 \bigg)\|(\mathcal T_\epsilon ^s)^{-1}\|_{p}
  \|b\|_{{\rm BMO}(bD, \lambda)}.
$$

\vskip0.1in

We next prove the  necessity.  Suppose that both $[b,\Sopo]$ and $[b,\Cine](I-\Sopo)$ are bounded on $L^p(bD, \omega)$ for some $1<p<\infty$ with $0<\epsilon<\epsilon (D)$.

From \eqref{E:dense1} we obtain that for any $f$  in the H\"older-like space \eqref{E:Holder},
\begin{align*}
 [ b, \Sopo ] T_\epsilon f &= [b , \EuScript C_\epsilon ](f)\, -\, \Sopo  [b, I-(\EuScript C^\dagger_\epsilon - \EuScript C_\epsilon) ](f) \\
 &= [b , \EuScript C_\epsilon ](f)\, +\, \Sopo  [b, \EuScript C^\dagger_\epsilon] (f) - \Sopo  [b,\EuScript C_\epsilon ](f) .
\end{align*}
Thus, we have
\begin{align}\label{eee1}
 (I- \Sopo)  [b,\EuScript C_\epsilon ](f)  
 &= [ b, \Sopo ] T_\epsilon f  -\Sopo  [b, \EuScript C^\dagger_\epsilon] (f).
\end{align}
To continue, observe that the basic identity
$$
(\Sopo)f= (\EuScript C_\epsilon\Sopo)f\quad \mbox{for any}\ f\in L^2(bD, \omega)$$
 grants that the following equality
 \begin{align}\label{E:cond2}
 [b, \EuScript C_\epsilon]\Sopo f = (I-\EuScript C_\epsilon)[b,  \Sopo ] f
\end{align}
is
valid
whenever $f$ is  in the H\"older-like space \eqref{E:Holder}. Now the righthand side of \eqref{E:cond2} extends to a bounded operator on  $L^p(bD, \omega)$ by the main  result of \cite{DLLWW} along with our assumption on $[b,  \Sopo ]$. Thus, $ [b, \EuScript C_\epsilon]\Sopo$ in the left-hand side of \eqref{E:cond2}  extends to 
a bounded operator on $L^p(bD, \omega)$. 

By the assumption that $[b,\Cine](I-\Sopo)$ is bounded on $L^p(bD, \omega)$
and the fact that 
  $$[b,\Cine] =[b,\Cine]\Sopo + [b,\Cine](I-\Sopo),$$ we obtain that $[b,\Cine]$ extends to 
a bounded operator on $L^p(bD, \omega)$ with the norm
\begin{align}\label{E:cond4}
&\|[b,\Cine]\|_{L^p(bD, \omega)\to L^p(bD, \omega) }\\
&\leq
\|[b,\Cine]\Sopo\|_{L^p(bD, \omega)\to L^p(bD, \omega) }+
\|[b,\Cine](I-\Sopo)\|_{L^p(bD, \omega)\to L^p(bD, \omega) }\nonumber\\
&\leq (1+ \|\Cine\|_{L^p(bD, \omega)\to L^p(bD, \omega) })
\|[b,\Sopo]\|_{L^p(bD, \omega)\to L^p(bD, \omega) }\nonumber\\
&\qquad+
\|[b,\Cine](I-\Sopo)\|_{L^p(bD, \omega)\to L^p(bD, \omega) }.\nonumber
\end{align}

We now denote by $$[b, \EuScript C_\epsilon]^\dagger: L^{p'}(bD, \omega)\to 
L^{p'}(bD, \omega)$$
the duality of $[b, \EuScript C_\epsilon]$.

Here the duality goes through the following sense: for every $f,g$ in the H\"older-like space \eqref{E:Holder}, we have that
\begin{align*}
\langle [b, \EuScript C_\epsilon] (f), g\rangle  =\langle  f, [b, \EuScript C_\epsilon]^\dag g\rangle.
\end{align*}
Note that the associated kernel of $[b, \EuScript C_\epsilon]$ is given by 
$$ T(w,z) = (b(z)-b(w))C_\epsilon(w,z),\qquad w\not=z.$$
We have that the kernel of $ [b, \EuScript C_\epsilon]^\dag$ is 
$$ T^\dag(w,z) = -(b(w)-b(z))\overline{C_\epsilon(w,z)} .$$

It follows by duality of \eqref{E:cond4} that
\begin{equation}\label{E:mapping-2}
[b, \EuScript C_\epsilon]^\dagger: L^{p'}(bD, \omega)\to 
L^{p'}(bD, \omega).
\end{equation}
is bounded and that
$\displaystyle{ \|[b, \EuScript C_\epsilon]^\dagger \|_{L^{p'}(bD, \omega)\to 
L^{p'}(bD, \omega)} 
 \leq \|[b, \EuScript C_\epsilon] \|_{L^{p}(bD, \omega)\to 
L^{p}(bD, \omega)}}. $

Since $[b, \EuScript C_\epsilon]^\dagger $ is bounded on $L^{p'}(bD,\omega)$ and $\varphi$ has  uniform positive upper and lower bounds, we consider the 
commutator
$[b, \EuScript C_\epsilon]^{\!*}$ with the kernel $T^*(w,z) = \varphi(z)^{{1\over p'}} T^\dag(w,z)\varphi^{{1\over p}}(w) $, i.e.,
$$[b, \EuScript C_\epsilon]^{\!*}(f)(z) =\int_{bD}T^*(w,z) f(w)d\lambda(w),\quad f\in L^{p'}(bD, \lambda).$$
Then we 
obtain that for $f\in L^{p'}(bD, \lambda)$,
\begin{align*}
\|[b, \EuScript C_\epsilon]^{\!*}(f)\|_{L^{p'}(bD, \lambda)}
&= \bigg( \int_{bD} \Big| \int_{bD} T^*(w,z) f(w) d\lambda(w) \Big|^{p'} d\lambda(z)  \bigg)^{1\over p'}\\
&= \bigg( \int_{bD} \Big| \int_{bD} \varphi(z)^{-{1\over p'}} T^*(w,z)\varphi^{-1}(w) f(w)\ \varphi(w) d\lambda(w) \Big|^{p'} \varphi(z)d\lambda(z)  \bigg)^{1\over p'}\\
&= \bigg( \int_{bD} \Big| \int_{bD}  T^\dag(w,z)\cdot \varphi^{-{1\over p'}}(w) f(w)\ \varphi(w) d\lambda(w) \Big|^{p'} \varphi(z)d\lambda(z)  \bigg)^{1\over p'}\\
&=\| [b, \EuScript C_\epsilon]^\dagger ( \varphi^{-{1\over p'}} f) \|_{ L^{p'}(bD, \omega) }\\
&\leq \|[b, \EuScript C_\epsilon]^\dagger \|_{L^{p'}(bD, \omega)\to 
L^{2}(bD, \omega)}\| \varphi^{-{1\over p'}} f \|_{ L^{p'}(bD, \omega) }\\
&= \|[b, \EuScript C_\epsilon]^\dagger \|_{L^{p'}(bD, \omega)\to 
L^{2}(bD, \omega)}\|  f \|_{ L^{p'}(bD, \lambda) },
\end{align*}
which implies that 
$ \|[b, \EuScript C_\epsilon]^{\!*} \|_{L^{2}(bD, \lambda)\to L^{2}(bD, \lambda)}  \lesssim  \|[b, \EuScript C_\epsilon]^\dagger \|_{L^{2}(bD, \omega)\to 
L^{2}(bD, \omega)} $.

Moreover, from the kernel of $[b, \EuScript C_\epsilon]^\dagger$, we further have that for $f$ with $z\not\in {\rm supp} f$, 
\begin{align*} 
[b, \EuScript C_\epsilon]^{\!*}(f)(z) &= \int_{bD} \varphi(z)^{{1\over p'}} T^\dag(w,z)\varphi^{{1\over p'}}(w) f(w)d\lambda(w) \\
&= \int_{bD} \varphi(z)^{{1\over p'}} (b(z)-b(w))\overline{C_\epsilon(w,z)} \varphi^{{1\over p'}}(w) f(w)d\lambda(w) \\
&= \int_{bD} \varphi(z)^{{1\over p'}} (b(z)-b(w))\overline{C^\sharp_\epsilon(w,z)} \varphi^{{1\over p'}}(w) f(w)d\lambda(w) \\
&\quad+ \int_{bD} \varphi(z)^{{1\over p'}} (b(z)-b(w))\overline{R_\epsilon(w,z)} \varphi^{{1\over p'}}(w) f(w)d\lambda(w) \\
&=: [b, (\EuScript C_\epsilon^\sharp)^*] (f)(z)+ [b, (\EuScript R_\epsilon)^*](f)(z),
\end{align*}
where the third equality follows from \eqref{E: C kernel}.

Recall that 
$$                   |C^\sharp_\epsilon(w,z)|\geq A_2 {\displaystyle1\over\displaystyle{\tt d}(w,z)^{2n}},
$$
and that
$$|R_\epsilon(w,z)|\leq C_R {\displaystyle1\over \displaystyle{\tt d}(w,z)^{2n-1}}.$$

As a consequence, we see that noting that $\varphi$ has  uniform positive upper and lower bounds, and by applying Theorem \ref{T:5.2}  to $[b, \EuScript C_\epsilon^{*}]$, we obtain that
 $b\in {\rm BMO}(bD, \lambda)$ with 
 $ \|b\|_{{\rm BMO}(bD, \lambda)}\lesssim \|[b, \EuScript C_\epsilon]^{\!*} \|_{L^{p'}(bD, \lambda)\to L^{p'}(bD, \lambda)} $, which further implies that
 \begin{align*}
   \|b\|_{{\rm BMO}(bD, \lambda)}\lesssim  \|[b, \EuScript C_\epsilon]^\dagger \|_{L^{p'}(bD, \omega)\to 
L^{p'}(bD, \omega)} \leq\|[b, \EuScript C_\epsilon] \|_{L^{p}(bD, \omega)\to 
L^{p}(bD, \omega)},
 \end{align*}
 where the implicit constant is independent of those $\epsilon$ in $(0,\epsilon(D))$ (see Theorem \ref{T:5.1}).

 Then, combining with \eqref{E:cond4}, we further have
 \begin{align*}
\|b\|_{{\rm BMO}(bD, \lambda)}&\lesssim  (1+ \|\Cine\|_{L^p(bD, \omega)\to L^p(bD, \omega) })
\|[b,\Sopo]\|_{L^p(bD, \omega)\to L^p(bD, \omega) }\nonumber\\
&\qquad+
\|[b,\Cine](I-\Sopo)\|_{L^p(bD, \omega)\to L^p(bD, \omega) }.\nonumber
 \end{align*}

The proof of Part (1) is concluded.
\end{proof}

\begin{proof}[Proof of Part \textnormal{(2)}] 
Suppose that
$b$ is in ${\rm VMO}(bD, \lambda)$. We claim that $[b, \Sopo]$  is compact on $L^2(bD, \omega)$.  This is immediate from \eqref{E:dense-2} which shows that  $[b, \Sopo]$ is the composition of compact operators (namely $[b , \EuScript C_\epsilon ]$ and $[b, T_\epsilon ]$, by Theorem \ref{T:5.2})
 with the operators $T^{-1}_\epsilon$ (which is bounded by the results of \cite{LS2017}) and $\Sopo$ (trivially bounded in $L^2(bD, \omega)$). The compactness in
 $L^p(bD, \omega)$ for $1<p<2$ follows by applying this same argument to the identity \eqref{E:dense-p}, once we point out that the extra term $[b, \Sopo]\mathcal R_\epsilon^s$ which occurs in the righthand side of \eqref{E:dense-p} is compact in $L^p(bD, \omega)$ on account of the compactness, just proved, of $[b, \Sopo]$ in $L^2(bD, \omega)$, and the chain of bounded inclusions \eqref{E:mapping-1};
 the compactness in the range $2<p<\infty$ now follows by duality. This concludes the proof of sufficiency.

\vskip0.1in
To prove the necessity, we suppose that $b\in {\rm BMO}(bD, \lambda)$, 
 $[b,\Sopo]$ and $[b,\Cine](I-\Sopo)$ are compact on $L^p(bD, \omega)$ for some $1<p<\infty$.
  We now prove that $b\in {\rm VMO}(bD, \lambda)$.

Since $[b,\Sopo]$ is compact on $L^p(bD, \omega)$ for some $1<p<\infty$, 
by \eqref{E:cond2}, we see that $[b, \Cine]\Sopo$  extends to a  compact operator on $L^p(bD, \omega)$.

This, together with the assumption that  $[b,\Cine](I-\Sopo)$ is compact on $L^p(bD, \omega)$,  further  shows that $[b,\EuScript C_\epsilon]$ is compact as an operator from 
$L^p(bD, \omega)\to L^p(bD, \omega)$ since it is the linear combination of compositions of a compact operator with the bounded operators. Thus 
$$
[b,\EuScript C_\epsilon]^\dagger: L^{p'}(bD, \omega)\to  L^{p'}(bD, \omega)
$$
is compact by duality. 

Following the argument at the end of the proof of Part (1), we see that this implies that 
$b\in {\rm VMO}(bD, \lambda)$ by
Theorem \ref{T:5.2}.

The proof of Theorem \ref{cauchy1} is concluded.
\end{proof}

\vskip0.1in
\subsection{The commutator of $\Sopo$: proof of Theorem \ref{cauchy4}} We may now proceed to study the behavior of the commutator $[b, \Sopo$ on the maximal $L^p$-spaces $L^p(bD, \Op)$. We prove all parts of Theorem \ref{cauchy4} one at a time. 

\vskip0.1cm
\noindent {\it Proof of Part \textnormal{(1)}}.
 We first prove the sufficiency. To this end,
it suffices to show that 
  \begin{equation}\label{E:extrapol-comm}
  \|[b, \Sopo]g\|_{L^2(bD, \Ot)}\lesssim  [\Ot]_{A_2}^2\|b\|_{{\rm BMO}(bD,\lambda)} \|g\|_{L^2(bD, \Ot)}
 \end{equation}
 holds for any $g\in C(bD)$ and for any $A_2$-like measure $\Ot$, where the implied constant depends only on $\omega$ and $D$, because the $L^p$-estimate \eqref{Lp-est-1}
  will then follow by extrapolation \cite[Section 9.5.2]{G}.
To prove \eqref{E:extrapol-comm}, for any $\epsilon >0$ we write
 $$
 [b, \Sopo]g = \tilde A_\epsilon g + \tilde B_\epsilon g + C_\epsilon g\quad \text{where}
 $$
 \begin{equation*}
 \tilde A_\epsilon g := 
 [b, \EuScript C_\epsilon]\circ\Tinv g\,;\quad 
 \tilde B_\epsilon g := 
-   [b, \Sopo]\circ\big((\EuScript R^s_\epsilon)^\dagger-\EuScript R^s_\epsilon\big)\circ 
 \Tinv g\, ,
  \end{equation*}
  and
  $$
  C_\epsilon g: = \Sopo\circ[b, I-\big((\EuScript R^s_\epsilon)^\dagger-\EuScript R^s_\epsilon\big)]\circ\Tinv g
  $$
  where again
  $$\mathcal T^s_\epsilon h := \bigg(I - \big((\EuScript C^s_\epsilon)^\dagger - \EuScript C^s_\epsilon\big)\bigg)h.$$
  
We first consider   $ \tilde A_\epsilon$.  By choosing $\epsilon = \epsilon (\Ot)$ as in the proof of \cite[Theorem 1.1]{DLLW} (see (4.1.) there), we see that $\Tinv$ is bounded on $L^2(bD, \Ot)$ with 
$\|\Tinv\|_{ L^2(bD, \Ot)\to L^2(bD, \Ot) }\leq2$.
Hence Theorem \ref{T:5.2} grants
\begin{align*}
  \| \tilde A_\epsilon g \|_{L^2(bD, \Ot)}\lesssim \|b\|_{{\rm BMO}(bD,\lambda)}\cdot {  [\Ot]_{A_2}^{4}} \cdot \|g\|_{L^2(bD, \Ot)}.
\end{align*}
  
To control the operator $\tilde B_\epsilon$, with same $\epsilon$ as above, it suffices to prove the boundedness of $[b, \Sopo]\circ\big((\EuScript R^s_\epsilon)^\dagger-\EuScript R^s_\epsilon\big)$. To this end, we combine
the mapping properties \eqref{E:RsBdd} with Part (1) of Theorem \ref{cauchy1} and the reverse H\"older's inequality, and obtain that
 \begin{align}\label{E:mapping-L2 norm}
[b, \Sopo]\circ\big((\EuScript R^s_\epsilon)^\dagger-\EuScript R^s_\epsilon\big) : 
L^2(bD, \Omega_2) & \hookrightarrow L^1(bD, \omega) \to L^\infty(bD, \omega)  \\\notag
&\hookrightarrow L^{2p_0}(bD, \omega) \to L^{2p_0}(bD, \omega) \hookrightarrow L^2(bD, \Omega_2);
\end{align}
here $p_0>2$ has been chosen so that its H\"older conjugate $p'_0$ satisfies 
$$ \bigg( \int_{bD} \Omega_2^{p'_0}(z)d\omega(z) \bigg)^{1\over p'_0} \leq M(D,\omega) \int_{bD} \Omega_2(z)d\omega(z)  = M(D,\omega)\,  \Omega_2(bD)\,,  $$
where the constant $M(D,\omega)$ {is independent of $\Omega_2$}.
Moreover, by writing $h:= \Tinv g$, $\overline H = \big((\EuScript R^s_\epsilon)^\dagger-\EuScript R^s_\epsilon\big) h$ and $\tilde B_\epsilon g =-[b, \Sopo] (\overline H)$,  we have that
\begin{align*}
\| \overline H \|_{ L^{2p_0}(bD, \omega) }&\leq \omega(bD)^{1\over 2p_0} \|\big((\EuScript R^s_\epsilon)^\dagger-\EuScript R^s_\epsilon\big) h\|_{L^\infty(bD, \omega)}\lesssim \omega(bD)^{1\over 2p_0} \|h\|_{L^1(bD, \omega)} \\
&\lesssim \omega(bD)^{1\over 2p_0} (\Omega^{-1}_2(bD))^{1\over2} \|h\|_{L^2(bD, \Ot)} 
\end{align*}
and that
\begin{align*}
\|[b, \Sopo] (\overline H)\|_{L^2(bD, \Ot)}\leq \|[b, \Sopo] (\overline H)\|_{ L^{2p_0}(bD, \omega) } \|\Omega_2\|_{L^{p'_0}(bD, \omega)}^{1\over2}\lesssim \|\overline H\|_{ L^{2p_0}(bD, \omega) } M(D,\omega)\,  \Omega_2(bD)^{1\over2}.
\end{align*}
Hence, we have
the norm 
\begin{align*}
  \| \tilde B_\epsilon g \|_{L^2(bD, \Ot)}  &\lesssim M(D,\omega) \Omega_2(bD)^{1\over2}  \|b\|_{{\rm BMO}(bD,\lambda)}(\Omega^{-1}_2(bD))^{1\over2} \|g\|_{L^2(bD, \Ot)}\\
  &\lesssim M(D,\omega) [\Omega_2]_{A_2}  \|b\|_{{\rm BMO}(bD,\lambda)}\|g\|_{L^2(bD, \Ot)},
\end{align*}
where the last inequality follows from the definition of the $A_2$ constant.

  To bound the norm of $C_\epsilon g$, we start by writing
  $$
  C_\epsilon g = \Sopo (\tilde H);\quad \tilde H:= [b, I - \big((\EuScript C^s_\epsilon)^\dagger - \EuScript C^s_\epsilon\big)]h;\quad
  h:= \Tinv g\, ,
  $$
  hence the conclusion of \cite[Theorem 1.1]{DLLW} (see (1.15) there) grants
  $$
  \|C_\epsilon g\|_{L^2(bD, \Ot)}\lesssim {  [\Ot]_{A_2}^3}\cdot \|\tilde H\|_{L^2(bD, \Ot)}.
  $$
  Furthermore, 
  $$
  \|\tilde H\|_{L^2(bD, \Ot)}\leq 
  \|
  [b, \EuScript C^s_\epsilon]h
  \|_{L^2(bD, \Ot)} +
   \|
  [b, (\EuScript C^s_\epsilon)^\dagger]h
  \|_{L^2(bD, \Ot)}.
  $$
  Now Theorem \ref{T:5.2} (for $p=2)$ with $\epsilon = \epsilon (\Ot)$ chosen as in the proof of \cite[Theorem 1.1]{DLLW} (see (4.1) there)
   gives that
  $$
  \| [b, \EuScript C^s_\epsilon]h\|_{L^2(bD, \Ot)}\leq C(\omega, D) \|b\|_{{\rm BMO}(bD, \lambda)}\cdot {  [\Ot]_{A_2}^4}\cdot \|h\|_{L^2(bD, \Ot)},
  $$
  and that
  $$
  \|h\|_{L^2(bD, \Ot)}\leq 2\|g\|_{L^2(bD, \Ot)}.
  $$
  Combining all of the above we obtain
  $$
  \| [b, \EuScript C^s_\epsilon]h\|_{L^2(bD, \Ot)}\leq 2C(\omega, D) \|b\|_{{\rm BMO}(bD, \lambda)} \cdot {  [\Ot]_{A_2}^4}\cdot\|g\|_{L^2(bD, \Ot)}.
  $$
  It now suffices to show that 
\begin{align}\label{claim b Cdag weighted}
\|
  [b, (\EuScript C^s_\epsilon)^\dagger]h
  \|_{L^2(bD, \Ot)}\leq C(\omega, D) \|b\|_{{\rm BMO}(bD, \lambda)} \cdot {  [\Ot]_{A_2}^4}\cdot\|h\|_{L^2(bD, \Ot)}.
\end{align}
 To see this, we first recall that from \cite[(5.7)]{LS2017}, $\EuScript C_\epsilon^s$ is given by 
$$\EuScript C_\epsilon^s(f)(z) = \EuScript C_\epsilon \big( f(\cdot) \chi_s(\cdot,z) \big)(z), \quad z\in bD$$
(see the proof of Proposition \ref{prop cancellation}). Recall also that $(\EuScript C^s_\epsilon)^\dagger= \varphi^{-1}(\EuScript C^s_\epsilon)^*\varphi$, where 
$\varphi$ is the density function of $\omega$ satisfying \eqref{E:LL-weights}.
Next, we observe that
 \begin{align*}
 [b, (\EuScript C^s_\epsilon)^\dagger](h)(x)
  &= b(x) \varphi^{-1}(x) (\EuScript C^s_\epsilon)^* (\varphi \cdot h)(x) -\varphi^{-1}(x) (\EuScript C^s_\epsilon)^* (b\cdot \varphi \cdot h)(x) \\
&= \varphi^{-1}(x) [b, (\EuScript C_\epsilon^s)^*]\Big(\varphi(\cdot)  h(\cdot) \Big)(x).
 \end{align*}
 
Thus, it suffices to show that $[b, (\EuScript C_\epsilon^s)^*]$ is bounded on $L^2(bD, \Ot)$. Assume that this is the case, then based on the fact that $\varphi$ is the density function of $\omega$ satisfying \eqref{E:LL-weights}, we obtain that
   \begin{align*}
 \|[b, (\EuScript C^s_\epsilon)^\dagger](h)\|_{L^2(bD, \Ot)} &=\Big \|\varphi^{-1} [b, (\EuScript C_\epsilon^s)^*]\Big(\varphi(\cdot)  h(\cdot) \Big)\Big\|_{L^2(bD, \Ot)}\\
 & \leq {M_{(D, \varphi)}\over m_{(D, \varphi)}} \|[b, (\EuScript C_\epsilon^s)^*]\|_{L^2(bD, \Ot)\to L^2(bD, \Ot)} \|h\|_{L^2(bD, \Ot)}.
 \end{align*}
 
Next, by noting that for any $\Omega_2\in A_2$, $f_1\in L^2(bD,\Omega_2)$ and $f_2\in L^2(bD, \Omega_2^{-1})$ (recall that $\Omega_2^{-1}$ is also an $A_2$ weight), we have
 \begin{align*}
\langle [b, (\EuScript C_\epsilon^s)^*](f_1), f_2 \rangle&= \int_{bD} [b, (\EuScript C_\epsilon^s)^*](f_1)(x) f_2(x)\, d\lambda(x)\\
&= \int_{bD} f_1(x)\ [b, (\EuScript C_\epsilon^s)^*]^*(f_2)(x) \, d\lambda(x)\\
&= \int_{bD} f_1(x) \psi_2^{1\over2}(x)\ [b, \EuScript C_\epsilon^s](f_2)(x)  \psi_2^{-{1\over2}}(x) \,d\lambda(x),
 \end{align*}
which gives that 
$\displaystyle{
|\langle [b, (\EuScript C_\epsilon^s)^*](f_1), f_2 \rangle|
\leq \|f_1\|_{L^2(bD,\Omega_2)}\| [b, \EuScript C_\epsilon^s](f_2) \|_{L^2(bD,\Omega_2^{-1})},
}$
 and therefore $$\|[b, (\EuScript C_\epsilon^s)^*]\|_{L^2(bD, \Ot)\to L^2(bD, \Ot)}\leq \|[b, \EuScript C_\epsilon^s]\|_{L^2(bD, \Ot^{-1})\to L^2(bD, \Ot^{-1})}. $$
Now Theorem \ref{T:5.2} (for $p=2)$ with $\epsilon = \epsilon (\Ot)$ chosen again as in the  proof of \cite[Theorem 1.1]{DLLW} gives that the right-hand side in the above inequality  is bounded by 
 $C(\omega, D) \|b\|_{{\rm BMO}(bD, \lambda)} { [\Ot^{-1}]_{A_2}^4}$, which, together with the fact that $[\Ot^{-1}]_{A_2}=[\Ot]_{A_2}$, leads to
    \begin{align*}
 \|[b, (\EuScript C^s_\epsilon)^\dagger](h)\|_{L^2(bD, \Ot)} 
 & \leq {M_{(D, \varphi)}\over m_{(D, \varphi)}} C(\omega, D) \|b\|_{{\rm BMO}(bD, \lambda)}{ [\Ot]_{A_2}^4} \|h\|_{L^2(bD, \Ot)}.
 \end{align*}

\vskip0.1in

We next prove the  necessity.  Suppose that $b\in L^2 (bD, \lambda)$ and that the commutator 
$[b,\Sopo]$ and $[b,\Cine](I-\Sopo)$ are bounded on $L^p(bD, \Omega_p)$ for some $1<p<\infty$ {and for some $A_p$-measure $\Op$} with the density function $\psi_p$. We aim to show that $b\in {\rm BMO}(bD, \lambda)$: {we will do so by proving that (for any arbitrarily fixed $0<\epsilon<\epsilon (D)$) the commutator $[b, \EuScript C_\epsilon^*]$ is bounded on $L^{p'}(bD, \Omega_{p'})$ where $1/p +1/p'=1;$  $\Op' := \Op^{-{1\over p-1}}$, and $\EuScript C_\epsilon^*$ is the $L^2(bD, \sigma)$-adjoint of $\EuScript C_\epsilon$; the desired conclusion will then follow by Theorem \ref{T:5.2}.

{ 
For every $f$ in the H\"older-like space \eqref{E:Holder}, 
by \eqref{E:cond2} we see that 
\begin{align*}
 &\|[b, \EuScript C_\epsilon]\Sopo \|_{L^p(bD, \Omega_p)\to L^p(bD, \Omega_p)}\\
 & \leq 
 \Big( 1+ \|\Cine \|_{L^p(bD, \Omega_p)\to L^p(bD, \Omega_p)} \Big) \|[b,  \Sopo ]\|_{L^p(bD, \Omega_p)\to L^p(bD, \Omega_p)}.\nonumber
\end{align*}
This, together with the assumption that $[b,\Cine](I-\Sopo)$ is bounded on $L^p(bD, \Omega_p)$, gives that
$[b,\Cine]$ is bounded on $L^p(bD, \Omega_p)$ with 
\begin{align}\label{E:cond4 weighted}
 &\|[b, \EuScript C_\epsilon] \|_{L^p(bD, \Omega_p)\to L^p(bD, \Omega_p)}\\
 & \leq 
 \Big( 1+ \|\Cine \|_{L^p(bD, \Omega_p)\to L^p(bD, \Omega_p)} \Big) \|[b,  \Sopo ]\|_{L^p(bD, \Omega_p)\to L^p(bD, \Omega_p)}\nonumber\\
 &\qquad+\|[b,  \Cine ](I-\Sopo)\|_{L^p(bD, \Omega_p)\to L^p(bD, \Omega_p)}.\nonumber
\end{align}
}

 Now the adjoints of $[b, \EuScript C_\epsilon]$ in $L^2(bD,\omega)$ and in $L^2(bD, \sigma)$ (respectively denoted by upper-scripts $\dagger$ and $*$) are related to one another via the identity
$ [b, \EuScript C_\epsilon]^\dagger = \varphi^{-1} [b, \EuScript C_\epsilon]^{\!*}\varphi$,
{where}
$\varphi$ and its reciprocal $\varphi^{-1}$ satisfy
\eqref{E:Leray Levi to sigma}.
Since $[b, \EuScript C_\epsilon]^\dagger $ is bounded on $L^{p'}(bD,\Omega'_p)$ 
and $\varphi$ has positive and finite upper and lower bounds on $bD$, we obtain that
$[b, \EuScript C_\epsilon]^{\!*} $ is also bounded on $L^{p'}(bD,\Omega'_p)$ and moreover,
$$  \|[b, \EuScript C_\epsilon]^{\!*} \|_{L^{p'}(bD, \Omega'_p)\to L^{p'}(bD, \Omega'_p)}  \lesssim  \|[b, \EuScript C_\epsilon]^\dagger \|_{L^{p'}(bD, \Omega'_p)\to L^{p'}(bD, \Omega'_p)}\leq\|[b, \EuScript C_\epsilon] \|_{L^p(bD, \Omega_p)\to L^p(bD, \Omega_p)}.  $$

{But  $[b, \EuScript C_\epsilon]^{\!*} = [b, \EuScript C_\epsilon^{*}]$, hence the conclusion $b\in \rm{BMO}(bD, \lambda)$} and the desired bound:
 \begin{align*}   \|b\|_{{\rm BMO}(bD, \lambda)}&\lesssim c_\epsilon [\Op]_{A_p}^{  \max\{1,{1\over p-1}\}} 
 \|[b,  \Sopo ]\|_{L^p(bD, \Omega_p)\to L^p(bD, \Omega_p)} \\
 &\qquad+\|[b,  \Cine ](I-\Sopo)\|_{L^p(bD, \Omega_p)\to L^p(bD, \Omega_p)}
 \end{align*}
 follow from Theorem \ref{T:5.1} and \eqref{E:cond4 weighted}.
The proof of Part (1) is concluded.

\begin{proof}[Proof of Part \textnormal{(2)}] This
  follows a similar approach to the proof of (2) of Theorem \ref{cauchy1}
with standard modifications which can be seen from the proof of (1) above and the extrapolation compactness on weighted Lebesgue spaces \cite{Hy1}; we omit the details.
\vskip0.1in
The proof of Theorem \ref{cauchy4} is complete. 
\end{proof}

 \medskip
\subsection{The commutator of $\SopO$: proof of Theorem \ref{cauchy2}} 
As before, the superscript $\spadesuit$ designates the adjoint with respect to the inner product $\langle \cdot,\cdot\rangle_{\Omega_2}$ of $L^2(bD,\Omega_2).$ 
Thus, $\SopO$ is the orthogonal projection of $L^2(bD,\Omega_2)$ onto $H^2(bD,\Omega_2)$ in the sense that
$$\SopO^\spadesuit = \SopO,$$
where $H^2(bD,\Omega_2)$ is the holomorphic Hardy
and the 
$\SopO^\spadesuit$ denotes the adjoint of $\SopO$ in $L^2(bD,\Omega_2)$.

To begin with, we first point out that if $b$ is in ${\rm BMO}(bD,\lambda)$, then $b$ is in $L^2(bD, \Omega_2)$,
where $\Omega_2$ has the density function $\psi\in A_2$. 
Then, following the result in \cite[Section 5.2]{HT} (see also \cite[Theorem 3.1]{Ho} in $\mathbb R^n$), we see that 
$$
{\rm BMO}(bD,\lambda)={\rm BMO}_{L^p_{\Omega_2}}(bD,\lambda)
$$ 
for all $1\leq p<\infty$ and the norms are mutually equivalent, where ${\rm BMO}_{L^p_{\Omega_2}}(bD,\lambda)$ is the space of all $b\in L^1(bD,\lambda)$ such that 
$$ 
\|b\|_{*,\Omega_2} :=\sup_B \bigg({1\over \Omega_2(B)}\int\limits_B |b(z)-b_B|^p d\Omega_2(z)   \bigg)^{1\over p}<\infty,\quad b_B= {1\over \lambda(B)}\int\limits_{B } b(w)d\lambda(w),
$$
and $\|b\|_{{\rm BMO}_{L^p_{\Omega_2}}(bD,\lambda)} = \|b\|_{*,\Omega_2} + \|b\|_{L^1(bD,\lambda)}$.  Since $bD$ is compact, we see that $b\in L^p(bD,\Omega_2)$ for $1\leq p<\infty$, that is
\begin{equation}\label{E:incl}
{\rm BMO}(bD, \lambda)\subset L^p(bD, \Omega_2)\quad \mathrm{for\ any}\ \ 1\leq p<\infty.
\end{equation}

We split the proof into two parts.

\smallskip
\noindent {\it Proof of Part \textnormal{(1)}}.
We first prove the sufficiency.   We suppose that $b$ is in ${\rm BMO}(bD,\lambda )$ and show that $[b, \EuScript  S_{\Omega_2}]:\ L^2(bD, \Omega_2)\to L^2(bD, \Omega_2)$ for every $\varphi \in A_2$ with 
\begin{equation}\label{E:N(s)}
 \|[b, \EuScript  S_{\Omega_2}]\|_{2}
 \lesssim N([\psi]_{A_2}), 
 \end{equation}
where $N(s)$ is a positive increasing function on $[1,\infty)$.

We start with the following basic identity 
\begin{align}\label{e4.1}
\EuScript  S_{\Omega_2}{\EuScript  C_\epsilon^\spadesuit} (f)
=   ( \EuScript  C_\epsilon \EuScript  S_{\Omega_2}^\spadesuit)^\spadesuit (f) =  (\EuScript  C_\epsilon \EuScript  S_{\Omega_2})^\spadesuit (f) = (\EuScript  S_{\Omega_2})^\spadesuit (f) = \EuScript  S_{\Omega_2} (f), 
\end{align}
which is valid for any $f\in L^2(bD,\Omega_2 )$ and for any $\epsilon$ (whose value is not important here).

It follows from \eqref{e4.1} that 
\begin{align}\label{e4.2}
\EuScript  S_{\Omega_2} [b,T_{\epsilon,\Omega_2}](f) = \EuScript  S_\omega bT_{\epsilon,\Omega_2} (f) =  \EuScript  C_\epsilon( b f),
\end{align}
where
$$T_{\epsilon,\Omega_2}: = I - (\EuScript  C_\epsilon^\spadesuit-\EuScript  C_\epsilon )$$
and 
$f$ is any function taken in the H\"older-like space \eqref{E:Holder}.  
On the other hand, the classical Kerzman--Stein identity  \cite{KS}
\begin{equation}\label{E:KS 4.3}
  \SopO T_{\epsilon,\Omega_2} f
  =\EuScript C_\epsilon f,\quad f\in L^2(bD, \Omega_2),
 \end{equation}
 gives that
  \begin{equation}\label{E:1 4.4}
  b  \SopO T_{\epsilon,\Omega_2} f= b \EuScript C_\epsilon f,\quad f\in L^2(bD, \Omega_2).
  \end{equation}
 Combining \eqref{e4.2} and \eqref{E:1 4.4}  we obtain
 \begin{equation}\label{E:dense1 4.5}
 [ b, \SopO ] T_{\epsilon,\Omega_2} f = \big([b , \EuScript C_\epsilon ]\, +\, \SopO  [b, T_{\epsilon,\Omega_2} ]\big)f 
 \end{equation}
 whenever $f$ is in the H\"older-like space \eqref{E:Holder}. We now point out that the righthand side of \eqref{E:dense1 4.5} is meaningful in $L^2(bD,\Omega_2)$ by the same argument as before. We observe here that
\begin{align}\label{b T e w} 
[b, T_{\epsilon,\Omega_2} ] =  [b, I] - [b, \EuScript  C_\epsilon^\spadesuit] +[b, \EuScript  C_\epsilon]=- [b, \EuScript  C_\epsilon^\spadesuit] +[b, \EuScript  C_\epsilon]
\end{align}
and by (i) in Theorem \ref{T:5.2},
we get that $ [b, T_{\epsilon,\Omega_2} ] $  is also bounded on $L^2(bD,\Omega_2)$.   

Furthermore, we have  that $T_{\epsilon,\Omega_2}$ is invertible in $L^2(bD, \Omega_2)$ by the analogous two facts as in the proof of Theorem \ref{cauchy1}: 
(1.), $\EuScript C_\epsilon$ and $\EuScript C_\epsilon^\spadesuit$ are bounded in $L^2(bD, \Omega_2)$ and (2.),
 $T_{\epsilon,\Omega_2}$ is skew adjoint (that is, $(T_{\epsilon, \Omega_2})^\spadesuit = - T_{\epsilon, \Omega_2}$).
 We conclude that
 \begin{equation}\label{E:dense2 4.6}
 [ b, \SopO ] g = \big([b , \EuScript C_\epsilon ]\, +\, \SopO  [b, T_{\epsilon,\Omega_2} ]\big)\circ T_{\epsilon,\Omega_2}^{-1} g,\quad g\in L^2(bD, \Omega_2).
 \end{equation}
But the righthand side of \eqref{E:dense2 4.6} is bounded in $L^2(bD, \Omega_2)$ and 
\begin{align}\label{bS L2norm 4.7}  
\|[b, \SopO]\|_{2}&\lesssim \|T_{\epsilon,\Omega_2} ^{-1}\|_{2}\  \|[b, \EuScript C_\epsilon]\|_{2}\ \big (1+ \|\SopO\|_{2} \big) \\
&\lesssim \|T_{\epsilon,\Omega_2} ^{-1}\|_{2} [\Op]_{A_2}^{2}\|b\|_{{\rm BMO}(bD,\lambda)},\nonumber
\end{align} 
where the last inequality follows from (i) in Theorem \ref{T:5.2} and the fact that
$ \|\SopO\|_{2}=1$ by the definition of $\SopO$.

Hence we see that \eqref{E:N(s)} holds with $N(s):= Cs^2$ and  $C:=\|T_{\epsilon,\Omega_2} ^{-1}\|_{2}\|b\|_{{\rm BMO}(bD,\lambda)}$.

\vskip0.1in

We next prove the necessity. Suppose that $b$ is in $L^2(bD,\lambda)$ and that the commutator $[b,\SopO]:L^2(bD,\Omega_2)\to L^2(bD,\Omega_2)$ is bounded.

Repeating the same steps in the proof of the necessity part in Theorem \ref{cauchy1}, we see that
$[b, \EuScript C_\epsilon]$ is bounded from $L^{2}(bD,\Omega_2 )$ to 
$L^{2}(bD,\Omega_2 )$
 with
\begin{align}\label{Cdag bd by S}
 \|[b, \EuScript C_\epsilon] \|_{2 } 
 &\lesssim  
\|I- \EuScript C_\epsilon\|_{2 } \|[b,  \Sopo  ]\|_{2 },
 \end{align}
 where $\|I- \EuScript C_\epsilon\|_{2 }<\infty$ follows from Theorem \ref{T:5.1}.
 
 Then, by using (i) in Theorem \ref{T:5.2} (simply noting that $b\in L^2(bD,\Omega_2 )$ implies that $b\in L^1(bD,\lambda )$ since $\Omega_2^{-1}(bD)<\infty$), we obtain that $b$ is in ${\rm BMO}(bD,\lambda)$ with 
$ \|b\|_{{\rm BMO}(bD,\lambda)}\lesssim  \|[b, \EuScript C_\epsilon] \|_{2}$,
which, together with \eqref{Cdag bd by S}, gives
$$ \|b\|_{{\rm BMO}(bD,\lambda)}\lesssim  \|I- \EuScript C_\epsilon\|_{2} \|[b,  \Sopo  ]\|_{ 2}.$$

\vskip0.1in

{\em Proof of Part \textnormal{(2)}}. To prove the sufficiency, we assume that
$b$ is in ${\rm VMO}(bD,\lambda)$ and we aim to prove that $[b, \SopO]$ is compact on $L^2(bD, \Omega_2)$.

In fact, the argument that $[b, \SopO]$  is compact on $L^2(bD, \Omega_2)$ is immediate from \eqref{E:dense2 4.6}, which shows that  $[b, \SopO]$ is the composition of compact operators (namely $[b , \EuScript C_\epsilon ]$ and $[b, T_{\epsilon,\Omega_2} ]$, by (ii) of Theorem \ref{T:5.2}) with the bounded operators $T^{-1}_{\epsilon,\Omega_2}$ (by the results of \cite{LS2017}) and $\SopO$.

\vskip0.1in

To prove the necessity, we suppose that $b\in {\rm BMO}(bD,\lambda)$ and that
 $[b,\SopO]$ is compact on $L^2(bD, \Omega_2)$, and we show that $b\in {\rm VMO}(bD,\lambda)$.
To this end, we note that \eqref{Cdag bd by S} shows that 
$$
[b,\EuScript C_\epsilon]: L^{2}(bD, \Omega_2)\to L^{2}(bD, \Omega_2)
$$
is compact. But this implies that $b\in {\rm VMO}(bD, \lambda)$ by (ii)
of Theorem \ref{T:5.2}.

\medskip
The proof of Theorem \ref{cauchy2} is concluded.
\qed

\vskip0.1in

\vskip0.1in
\vskip0.1in
\vskip0.1in

{\bf Acknowledgements.}
X. Duong and Ji Li are supported by the Australian Research Council (award no. DP 220100285). L. Lanzani is 
a member of the INdAM group GNAMPA. 
B.\,D. Wick
is partially supported by the National Science Foundation (award no. DMS--1800057) and the Australian Research Council, award no. DP 220100285. We thank the mathematical research institute MATRIX in Australia where part of this research was performed.

\vskip0.1in
\bibliographystyle{amsplain}

\end{document}